\documentclass[10pt]{amsart}

\usepackage{hyperref}
\usepackage{amssymb,amsmath,amsthm,amsfonts,amsaddr}
\usepackage{comment}
\usepackage{tikz}

\newcommand{\class}{\mathsf}
\newcommand{\alg}{\mathbf}
\newcommand{\lang}{\mathcal}
\newcommand{\filter}{\mathcal}
\newcommand{\ideal}{\mathcal}
\newcommand{\assign}{:=}

\newcommand{\pair}[2]{\langle #1, #2 \rangle}
\newcommand{\triple}[3]{\langle #1, #2, #3 \rangle}
\newcommand{\set}[2]{\{ #1 \mid #2 \}}

\newcommand{\modframe}{\mathcal}

\newcommand{\AdjDiamond}{\blacklozenge}
\newcommand{\AdjBox}{\rule{1.5ex}{1.5ex}}

\newcommand{\minus}{-}
\newcommand{\sqleq}{\sqsubseteq}

\newcommand{\unit}{\eta}

\newcommand{\complex}[1]{#1^{\ast}}		% complex algebra
\newcommand{\canframe}[1]{#1_{\bullet}}	% canonical frame
\newcommand{\converse}[1]{\overline{#1}}	% relational converse

\theoremstyle{plain}
\newtheorem{theorem}{Theorem}[section]

\newtheorem{proposition}[theorem]{Proposition}
\newtheorem{corollary}[theorem]{Corollary}

\theoremstyle{definition}
\newtheorem{definition}[theorem]{Definition}
\newtheorem{example}[theorem]{Example}

\title{Compatibility between modal operators in distributive modal logic}
\author{Adam P\v{r}enosil}
\address{Institute of Computer Science, Czech Academy of Sciences}
\email{adam.prenosil@gmail.com}
\thanks{This research was supported by the grant 22-00137S of the Czech Science Foundation.}

\begin{document}

\begin{abstract}
  Unlike in classical modal logic, in non-classical modal logics the box and diamond operators frequently fail to be interdefinable. Instead, these logics impose some compatibility conditions which tie the box and the diamond together and ensure that in terms of Kripke semantics they arise from the same accessibility relation. This is the case in the intuitionistic modal logic of Fischer Servi as well as the positive modal logic of Dunn. In these logics, however, such compatibility conditions also impose further restrictions on the accessibility relation. In this paper, we identify the basic compatibility conditions which ensure that modal operators arise from a single accessibility relation without imposing any restrictions on the relation. As in the distributive logic of Gehrke, Nagahashi, and Venema, we allow for negative box and diamond operators here in addition to the usual positive ones. Intuitionistic modal logic and positive modal logic, or more precisely the corresponding classes of algebras, are then obtained in a modular way by adding certain canonical axioms which we call locality conditions on top of these basic compatibility conditions.
\end{abstract}

\maketitle

  Given a modal box operator $\Box$ and a modal diamond operator $\Diamond$ in a logic with distributive lattice connectives, how can we tell whether these represent two distinct modalities or different aspects of a single modality? In algebraic terms: given a box operator $\Box$ (a unary operation preserving finite meets) and a diamond operator $\Diamond$ (a unary operation preserving finite joins) on a distributive lattice $\alg{A}$, how can we tell whether they arise from a single accessibility relation on the dual space of $\alg{A}$?

  In the case of classical modal logic, this question has an easy answer: they arise from a single accessibility relation if and only if they are interdefinable as $\Diamond x = \neg \Box \neg x$, or equivalently $\Box x = \neg \Diamond \neg x$. If the box and diamond operators are no longer interdefinable, as happens in the case of the intuitionistic modal logic of Fischer Servi~\cite{fischer-servi77,fischer-servi84} (see \cite{simpson94} for a comprehensive discussion of Fischer Servi's logic and its relation to other variants of intuitionistic modal logic) and the positive modal logic of Dunn~\cite{dunn95} (the negation-free fragment of classical modal logic), the questios becomes substantially more complicated. In modal Heyting algebras (the algebraic semantics of intuitionistic modal logic) it is the compatibility relations
\begin{align*}
  & \Box x \wedge \Diamond y \leq \Diamond (x \wedge y), & \Diamond x \rightarrow \Box y & \leq \Box (x \rightarrow y),
\end{align*}
  which ensure that box and diamond arise from a single accessibility relation. In positive modal algebras (the algebraic semantics of positive modal logic) this role is played by the compatibility relations
\begin{align*}
  & \Box \wedge \Diamond y \leq \Diamond (x \wedge y), & \Box (x \vee y) & \leq \Box x \vee \Diamond y.
\end{align*}

  However, both of these pairs of equations impose further requirements above and beyond the fact that $\Box$ and $\Diamond$ come from the same relation, tying the accessibility relation $R$ to the partial order $\leq$ on the dual space in various ways. In intuitionistic modal logic, the conditions ${\geq} \circ R \subseteq R \circ {\geq}$ and $R \circ {\leq} \subseteq {\leq} \circ R$ are presupposed, while positive modal logic imposes the conditions ${\geq} \circ R \subseteq R \circ {\geq}$ and ${\leq} \circ R \subseteq R \circ {\leq}$.\footnote{We are using the partially ordered Kripke semantics of Celani and Jansana~\cite{celani+jansana97,celani+jansana99} here, rather than Dunn's original Kripke semantics.}

  In the present paper, we prove Kripke completeness theorems for logics whose modal operators arise from a single binary relation $R$ on a poset ordered by $\leq$ without imposing extra conditions relating $R$ and $\leq$. Such additional postulates can then be added \emph{\`{a} la carte} (see Example~\ref{example}), depending on one's choice of the modal signature (we allow for negative as well as positive modalities) and one's choice of desired compatibility conditions between $R$ and $\leq$ (we consider all $8$ possible conditions of the above kind). This in particular places both intuitionistic modal logic and positive modal logic within a systematic, modular framework as extensions of a single basic modal logic by locality conditions.

  We shall, however, not explicitly discuss logics \textit{qua} consequences relations or sets of formulas in this paper, choosing instead to deal directly with the corresponding classes of algebras. Associating a logic with the classes of algebras discussed in this paper is a mechanical task which has no bearing on our main concern here, namely determining the appropriate compatibility conditions between modal operators.

  The present paper is based on the conference paper \cite{prenosil14} and expands on it in several ways. Firstly, in addition to the usual box and diamond we allow for negative modal operators (as discussed in the next section). Secondly, we consider a wider range of locality conditions. And finally, we also consider in more detail the relationship between the two types Kripke semantics for distributive modal logic (namely what we-call multimodal and unimodal semantics below).

  The paper is structured as follows. In Section~\ref{sec:preliminaries} we review the work of Gehrke, Nagahashi, and Venema~\cite{gehrke+nagahashi+venema05} on distributive modal logic and recast it in our own notation and terminology. In Section~\ref{sec:unimodal} we discuss the relationship between the two possible semantics for extensions of distributive modal logic, namely multimodal and unimodal semantics. We axiomatize the quasivariety of modal algebras corresponing to the unimodal semantics. That is, we identify the minimal compatibility conditions which ensure that modal operators come from a single accessibility relation. Finally, in Section~\ref{sec:locality} we consider eight local modality conditions which bind together the order $\leq$ and the binary accessibility relation $R$ in a unimodal frame and show how to axiomatize the variety of modal algebras generated by the complex algebras of unimodal Kripke frames satisfying any combination of such conditions. This is a Kripke completeness result which in particular covers the known cases of intuitionistic modal logic and positive modal logic, as these are extensions of our basic modal logic by locality conditions. Example~\ref{example} illustrates how one picks the correct conditions from our menu of axioms when axiomatizing a particular logic.

\section{$\Delta_{\alpha}$-modal frames and algebras}
\label{sec:preliminaries}

  This section reviews the framework of distributive modal logic of Gehrke, Nagahashi, and Venema~\cite{gehrke+nagahashi+venema05}. We recast their definitions and results in slightly different terminology and notation. In particular, we use the symbols $\Diamond_{-}$ and $\Box_{-}$ instead of their symbols $\vartriangleleft$ and $\vartriangleright$, respectively. We also consider modalities which are adjoint to these, i.e.\ backward-looking modal operators in terms of their Kripke semantics. Apart from some simple observations about adjoint modalities, the results presented in this section are all taken from~\cite{gehrke+nagahashi+venema05}.

  We shall assume basic familiarity with universal algebra, distributive lattices, and Heyting algebras, which the reader may obtain from standard textbooks such as \cite{burris+sankappanavar81} and \cite{davey+priestley02}. Throughout the paper, by distributive lattices and their homomorphisms we shall always mean \emph{bounded} distributive lattices and homomorphisms which preserve these bounds.

  The algebras studied in this paper all have a distributive lattice reduct, some have a Heyting algebra reduct, and some in fact have a bi-Heyting algebra reduct. Let us recall the definition of bi-Heyting algebras.

\begin{definition}
  A \emph{Heyting algebra} is an algebra of the form $\alg{A} = \langle A, \wedge, \vee, \top, \bot \rightarrow \rangle$ such that $\langle A, \wedge, \vee, \top, \bot \rangle$ is a bounded distributive lattice and for all $a, b, c \in \alg{A}$
\begin{align*}
  a \wedge b \leq c & \iff b \leq a \rightarrow c.
\end{align*}
  A \emph{co-Heyting algebra} is an algebra of the form $\alg{A} = \langle A, \vee, \wedge, \bot, \top, \minus \rangle$ such that $\langle A, \wedge, \vee, \top, \bot \rangle$ is a bounded distributive lattice and for all $a, b, c \in \alg{A}$
\begin{align*}
  a \leq b \vee c & \iff a \minus b \leq c.
\end{align*}
  A \emph{bi-Heyting algebra} is an algebra of the form $\alg{A} = \langle A, \wedge, \vee, \top, \bot, \rightarrow, \minus \rangle$ such that $\langle A, \wedge, \vee, \top, \bot \rightarrow \rangle$ is a Heyting algebra and $\langle A, \wedge, \vee, \top, \bot, \minus \rangle$ is a co-Heyting algebra.
\end{definition}

  A \emph{polarity} is simply some $\alpha \in \{ +, - \}$. The variables $\alpha$ and $\beta$ will be reserved for polarities. For each poset $\pair{U}{\leq}$ and polarity $\alpha$ we define the poset $\pair{U}{\leq}^{\alpha}$ as follows:
\begin{align*}
  \pair{U}{\leq}^{\alpha} \assign \begin{cases} & \pair{U}{\leq} \text{ if } \alpha = +, \\ & \pair{U}{\geq} \text{ if } \alpha = -.\end{cases}
\end{align*}
   Likewise, given a distributive lattice $\alg{L} = \langle L, \wedge, \vee, \top, \bot \rangle$, we define $\alg{L}^{\alpha}$ as follows:
\begin{align*}
  \alg{L}^{\alpha} \assign \begin{cases} & \langle L, \wedge, \vee, \top, \bot \rangle \text{ if } \alpha = +, \\ & \langle L, \vee, \wedge, \bot, \top \rangle \text{ if } \alpha = -.\end{cases}
\end{align*}

  A \emph{modal symbol} $\Delta_{\alpha}$ is a pair consisting of some $\Delta \in \{ \Box, \Diamond \}$ and a polarity $\alpha \in \{ +, - \}$. In the following, we use $\Delta$ and $\nabla$ as variables which stand for either of the symbols $\Box$ and $\Diamond$ (or their adjoints, introduced below). The modal symbol \emph{opposite} to $\Box_{\alpha}$ (to $\Diamond_{\alpha}$) is the symbol $\Diamond_{\alpha}$ ($\Box_{\alpha}$).

\begin{definition}
  Let $\pair{U}{\leq}$ be a poset. We define the poset $\pair{U}{\leq} \times_{\Delta_{\alpha}} \pair{U}{\leq}$ as follows:
\begin{align*}
  \pair{U}{\leq} \times_{\Delta_{\alpha}} \pair{U}{\leq} \assign \begin{cases} & \pair{U}{\leq}^{-} \times \pair{U}{\leq}^{\alpha} \text{ if } \Delta = \Box, \\ & \pair{U}{\leq} \times \pair{U}{\leq}^{-\alpha} \text{ if }\Delta = \Diamond. \end{cases}
\end{align*}
  We say that a binary relation $R \subseteq U \times U$ is \emph{$\Delta_{\alpha}$-monotone} or simply that it is a \emph{$\Delta_{\alpha}$-relation} on $\pair{U}{\leq}$ if it is an upset of $\pair{U}{\leq} \times_{\Delta_{\alpha}} \pair{U}{\leq}$. A \emph{$\Delta_{\alpha}$-modal frame} is then a poset equipped with a $\Delta_{\alpha}$-relation $R^{\Delta}_{\alpha}$.
\end{definition}

  The least $\Delta_{\alpha}$-relation which extends $R$ will be denoted $\Delta_{\alpha}[R]$:
\begin{align*}
  \Box_{+}[R] & \assign {\leq} \circ {R} \circ {\leq}, & \Diamond_{+}[R] & \assign {\geq} \circ {R} \circ {\geq}, \\
  \Box_{-}[R] & \assign {\leq} \circ {R} \circ {\geq}, & \Diamond_{-}[R] & \assign {\geq} \circ {R} \circ {\leq}.
\end{align*}

\begin{definition}
  The \emph{complex algebra} $\complex{\modframe{F}}$ of a $\Delta_{\alpha}$-modal frame $\modframe{F} = \triple{U}{\leq}{R^{\Delta}_{\alpha}}$ is the distributive lattice of all upsets of $\pair{U}{\leq}$ equipped, depending on $\Delta_{\alpha}$, with one of the following operators:
\begin{align*}
  \Box_{+} a & \assign \set{u \in W}{u R^{\Box}_{+} v \text{ implies } v \in a}, & \Diamond_{+} a & \assign \set{u \in W}{u R^{\Diamond}_{+} v \text{ for some } v \in a}, \\
  \Box_{-} a & \assign \set{u \in W}{u R^{\Box}_{-} v \text{ implies } v \notin a}, & \Diamond_{-} a & \assign \set{u \in W}{u R^{\Diamond}_{-} v \text{ for some } v \notin a}.
\end{align*}
\end{definition}

  We use the notation $\complex{\modframe{F}}$ for the complex algebra rather than the more common $\modframe{F}^{+}$ in other to avoid confusion with other uses of $+$ as a superscript or subscript.

  Such algebras are canonical examples of what we call $\Delta_{\alpha}$-modal algebras.

\begin{definition} \label{def:modal-algebra}
  Let $\alg{L} = \langle L, \wedge, \vee, \top, \bot \rangle$ be a distributive lattice. A \emph{box operator of polarity $\alpha$} on $\alg{L}$ is a homomorphism of unital meet semilattices $\Box_{\alpha}\colon \alg{L}^\alpha \rightarrow \alg{L}$. A \emph{diamond operator of polarity $\alpha$} is a homomorphism of unital join semilattices $\Diamond_{\alpha}\colon \alg{L}^\alpha \rightarrow \alg{L}$. A \emph{$\Box_{\alpha}$-modal algebra over $\alg{L}$} is an expansion of $\alg{L}$ by a box operator of polarity $\alpha$. A \emph{$\Diamond_{\alpha}$-modal algebra} is an expansion of $\alg{L}$ by a diamond operator of polarity $\alpha$.
\end{definition}

  The class of all $\Delta_{\alpha}$-modal algebras forms a variety axiomatized relative to the variety of distributive lattices by one of the following pairs of the equations, depending on $\Delta_{\alpha}$:
\begin{align*}
  \Box_{+} (a \wedge b) & = \Box_{+} a \wedge \Box{+} b, & \Box_{+} \top & = \top, & \Diamond_{+} (a \vee b) & = \Diamond_{+} a \vee \Diamond_{+} b, & \Diamond_{+} \bot = \bot, \\
  \Box_{-} (a \vee b) & = \Box_{-} a \wedge \Box{-} b, & \Box_{-} \bot & = \top, & \Diamond_{-} (a \wedge b) & = \Diamond_{-} a \vee \Diamond{-} b, & \Diamond_{-} \top = \bot.
\end{align*}
  The complex algebras of $\Delta_{\alpha}$-modal frames are precisely the \emph{perfect} $\Delta_{\alpha}$-modal algebras, i.e.\ $\Delta_{\alpha}$-modal algebras $\alg{A}$ such that both $\alg{A}$ and $\alg{A}^{-}$ are algebraic lattices and the map $\Delta_{\alpha}\colon \alg{A}^{\alpha} \to \alg{A}$ preserves arbitrary meets if $\Delta = \Box$ and arbitrary joins if $\Delta = \Diamond$. In particular, each finite modal algebra is perfect.

  Each perfect $\Delta_{\alpha}$-modal algebra may be expanded to a unique Heyting (bi-Heyting) $\Delta_{\alpha}$-modal algebra. Here by a \emph{Heyting (bi-Heyting) $\Delta_{\alpha}$-modal algebra} we simply mean an algebra which is both a Heyting (bi-Heyting) algebra and a $\Delta_{\alpha}$-modal algebra.

\begin{proposition}
  The complex algebras of $\Delta_{\alpha}$-modal frames are precisely the perfect $\Delta_{\alpha}$-modal algebras.
\end{proposition}

  Each modal algebra embeds into a perfect modal algebra in a natural way.

\begin{definition}
  Let $\alg{A}$ be a $\Delta_{\alpha}$-modal algebra. The \emph{canonical frame} of $\alg{A}$, denoted $\canframe{\alg{A}}$, is the poset of all prime filters on $\alg{A}$ ordered by inclusion equipped with the $\Delta_{\alpha}$-relation $R^{\Delta}_{\alpha}$ defined, depending on $\Delta_{\alpha}$, as:
\begin{align*}
  \filter{U} R^{\Box}_{+} \filter{V} & \iff \text{ $\Box_{+} a \in \filter{U}$ implies $a \in \filter{V}$}, \\
  \filter{U} R^{\Box}_{-} \filter{V} & \iff \text{ $\Box_{-} a \in \filter{U}$ implies $a \notin \filter{V}$}, \\
  \filter{U} R^{\Diamond}_{+} \filter{V} & \iff \text{ $a \in \filter{V}$ implies $\Diamond_{+} a \in \filter{U}$}, \\
  \filter{U} R^{\Diamond}_{-} \filter{V} & \iff \text{ $a \notin \filter{V}$ implies $\Diamond_{-} a \in \filter{U}$}.
\end{align*}
\end{definition}

  We define the map $\unit_{\alg{A}}\colon \alg{A} \rightarrow \complex{(\canframe{\alg{A}})}$ as $\unit_{\alg{A}} (a) \assign \set{\filter{U} \in \canframe{\alg{A}}}{a \in \filter{U}}$.

\begin{theorem}[Canonical embedding] \label{thm:canonical-embedding}
  Let $\alg{A}$ be a modal $\Delta_{\alpha}$-algebra. Then the map $\unit_{\alg{A}}\colon \alg{A} \rightarrow \complex{(\canframe{\alg{A}})}$ is an embedding of modal $\Delta_{\alpha}$-algebras called the \emph{canonical embedding}.
\end{theorem}

  The algebra $\complex{(\canframe{\alg{A}})}$ and the embedding $\unit_{\alg{A}}$ may also be introduced algebraically without mentioning the canonical frame of $\alg{A}$. The map $\unit_{\alg{A}}$ is, in a suitable sense, a dense and compact embedding of $\alg{A}$ into $\complex{(\canframe{\alg{A}})}$, and the algebra $\complex{(\canframe{\alg{A}})}$, also called the \emph{canonical extension} of $\alg{A}$, is up to isomorphism the unique algebra into which $\alg{A}$ embeds in a dense and compact way. This is the perspective adopted in~\cite{gehrke+nagahashi+venema05}. However, we shall be satisfied with the less abstract definition in terms of the canonical frame of $\alg{A}$.

  A class of modal $\Delta_{\alpha}$-algebras closed under canonical extensions will be called \emph{canonical}. By extension, a set of universal sentences is canonical if they axiomatize a canonical universal class. Theorem \ref{thm:canonical-embedding} then implies that each canonical universal class is generated by its perfect algebras. Being generated by perfect algebras is precisely the algebraic counterpart of being complete with respect to some class of modal frames, i.e.\ the algebraic formulation of being \emph{Kripke complete}.

  We shall also consider modalities which are \emph{adjoint} to the four modal operators introduced above. We shall not not study algebras with adjoint modalities for their own sake, but they will prove useful in understanding what we call local modality conditions. Semantically, adjoint modalities governed by the \emph{converse} of the modal accessibility relations governing ordinary modalities. (The \emph{converse} of a relation $R$ is the relation $\converse{R}$ such that $x \converse{R} y$ if and only if $y R x$.) They are accordingly also called \emph{backward} modalities and denoted $\AdjBox_{\alpha}$ and $\AdjDiamond_{\alpha}$, in contrast to the \emph{forward} modalities $\Box_{\alpha}$ and $\Diamond_{\alpha}$. Each forward modality has a corresponding adjoint backward modality, the adjoint pairs being
\begin{align*}
  & \Box_{+} \text{ and } \AdjDiamond_{+}, & & \Box_{-} \text{ and } \AdjBox_{-}, \\
  & \Diamond_{+} \text{ and } \AdjBox_{+}, & & \Diamond_{-} \text{ and } \AdjDiamond_{-}.
\end{align*}

  A \emph{$\AdjBox_{\alpha}$-modal ($\AdjDiamond_{\alpha}$-modal) frame} or \emph{algebra} is simply a $\Box_{\alpha}$-modal ($\Diamond_{\alpha}$-modal) frame or algebra under a different name.  The converse of each $\Delta_{\alpha}$-relation is a $\nabla_{\!\alpha}$-relation, where $\nabla_{\!\alpha}$ is the modality adjoint to~$\Delta_{\alpha}$. For example, the converse of a $\Box_{+}$-relation is a $\AdjDiamond_{+}$-relation (i.e.\ a $\Diamond_{+}$-relation denoted by $\AdjDiamond_{+}$) and the converse of a $\Box_{-}$ relation is a $\AdjBox_{-}$ relation (i.e\ a $\Box_{-}$-relation denoted by $\AdjBox_{-}$).

\begin{definition}
  A \emph{tense $\Delta_{\alpha}$-modal frame} is both a $\Delta_{\alpha}$-modal frame and a $\nabla_{\!\alpha}$-modal frame over the same poset, where $\nabla_{\!\alpha}$ is adjoint to $\Delta_{\alpha}$, such that $R^{\Delta}_{\alpha}$ and $R^{\nabla}_{\alpha}$ are mutually converse relations.
\end{definition}

  For example, a \emph{tense $\Box_{+}$-modal frame} is a structure $\langle U, \leq, R^{\Box}_{+}, R^{\AdjDiamond}_{+} \rangle$ such that $R^{\Box}_{+}$ is a $\Box_{+}$-relation on $\pair{U}{\leq}$ and $R^{\AdjDiamond}_{+}$ is the converse $\AdjDiamond_{+}$-relation on $\pair{U}{\leq}$. The \emph{complex algebra} of a tense modal frame is accordingly a distributive lattice expanded by two modal operators. 

\begin{definition}

  A \emph{tense $\Delta_{\alpha}$-modal algebra} is both a $\Delta_{\alpha}$-modal algebra a $\nabla_{\!\alpha}$-modal algebra over the same distributive lattice such that the symbol $\nabla_{\!\alpha}$ is adjoint to $\Delta_{\alpha}$ and the appropriate pair of the following conditions holds:
\begin{align*}
  & a \leq \Box_{+} \AdjDiamond_{+} a, & & a \leq \AdjBox_{+} \Diamond_{+} a, & & a \leq \Box_{-} \AdjBox_{-} a, & & \AdjDiamond_{-} \Diamond_{-} a \leq a, \\
  & \AdjDiamond_{+} \Box_{+} a \leq a, & & \Diamond_{+} \AdjBox_{+} a \leq a, & & a \leq \AdjBox_{-} \Box_{-} a, & & \Diamond_{-} \AdjDiamond_{-} a \leq a.
\end{align*}
  Equivalently, the appropriate equivalence below holds:
\begin{align*}
  a \leq \Box_{+} b & \iff \AdjDiamond_{+} a \leq b, &
  \Diamond_{+} a \leq b & \iff a \leq \AdjBox_{+} b, \\
  a \leq \Box_{-} b & \iff b \leq \AdjBox_{-} a, &
  \Diamond_{-} a \leq b & \iff \AdjDiamond_{-} b \leq a.
\end{align*}
\end{definition}
  For example, a tense $\Box_{+}$-modal algebra $\alg{A}$ is both a $\Box_{+}$-modal algebra and a $\AdjDiamond_{+}$-modal algebra over the same distributive lattice such that $a \leq \Box_{+} \AdjDiamond_{+} a$ and $\AdjDiamond_{+} \Box_{+} a \leq a$ hold for all $a \in \alg{A}$. A tense $\Delta_{\alpha}$-modal algebra is \emph{perfect} if it is perfect both as a $\Delta_{\alpha}$-modal algebra and as a $\nabla_{\!\alpha}$-modal algebra.

\begin{proposition}
  The complex algebras of tense $\Delta_{\alpha}$-modal frames are the perfect tense $\Delta_{\alpha}$-modal algebras.
\end{proposition}

  The canonical frame of a tense $\Delta_{\alpha}$-modal algebra is defined in the expected way as a pair of canonical frames over the same poset. Crucially, the adjointness conditions above ensure that the canonical frame is indeed a tense modal frame. Equivalently, we may say that the adjointness conditions are canonical.

\begin{proposition}
  The canonical frame of a tense $\Delta_{\alpha}$-modal algebra is a tense $\Delta_{\alpha}$-modal frame.
\end{proposition}

  Finally, it will be useful to make explicit two symmetries present in our system of four forward modalities and four backward modalities. Firstly, we shall say that the modalities $\Box_{\alpha}$ and $\AdjBox_{\alpha}$ are \emph{opposite} to the modalities $\Diamond_{\alpha}$ and $\AdjDiamond_{\alpha}$. This relation will also be called \emph{order duality}. Secondly, we have the adjointness relation between the appropriate forward and backward modalities, which we call \emph{forward--backward duality}. Exploiting these two symmetries will often allow us to cut our work down by half.

  These symmetries extend to the semantic and algebraic structures introduced above. The order dual of a $\Delta_{\alpha}$-modal frame $\modframe{F} = \triple{U}{\leq}{R^{\Delta}_{\alpha}}$ is the $\nabla_{\!\alpha}$-modal frame $\modframe{F}^{-} = \triple{U}{\geq}{R^{\Delta}_{\alpha}}$  and the order dual of a $\Delta_{\alpha}$-modal algebra $\alg{A} = \langle A, \wedge, \vee, \top, \bot, \Delta_{\alpha} \rangle$ is the $\nabla_{\!\alpha}$-modal algebra $\alg{A}^{-} = \langle A, \vee, \wedge, \bot, \top, \Delta_{\alpha} \rangle$, where $\nabla_{\!\alpha}$ is the modality order dual to $\Delta_{\alpha}$. The converse of a $\Delta_{\alpha}$-modal frame $\modframe{F} = (U, \leq, R^{\Delta}_{\alpha})$ is the $\nabla_{\!\alpha}$-modal frame $\converse{\modframe{F}} = (U, \leq, \overline{R^{\Delta}_{\alpha}})$ with the converse accessibility relation, where $\nabla_{\!\alpha}$ is the modality adjoint to $\Delta_{\alpha}$. Finally, the adjoint of a $\Delta_{\alpha}$-modal algebra $\alg{A} = \langle A, \wedge, \vee, \top, \bot, \Delta_{\alpha} \rangle$, if it exists, is the $\nabla_{\!\alpha}$-modal algebra $\converse{\alg{A}} = \langle A, \wedge, \vee, \top, \bot, \nabla_{\!\alpha} \rangle$ such that $\langle A, \wedge, \vee, \top, \bot, \Delta_{\alpha}, \nabla_{\!\alpha} \rangle$ is a tense $\Delta_{\alpha}$-modal algebra, where $\nabla_{\alpha}$ is the modality adjoint to $\Delta_{\alpha}$. These symmetries extend to tense modal frames and algebras in the expected way, i.e.\ componentwise. For example, the adjoint of a tense $\Box_{+}$-modal algebra $\langle A, \wedge, \vee, \top, \bot, \Box_{+}, \AdjDiamond_{+} \rangle$ is the tense $\Diamond_{+}$-modal algebra $\langle A, \wedge, \vee, \top, \bot, \Diamond_{+}, \AdjBox_{+} \rangle$ such that $\Diamond_{+} a = \AdjDiamond_{+} a$ and $\AdjBox_{+} a = \Box_{+} a$.

  Conveniently, the operations of taking the complex algebra of a modal frame and taking the canonical frame of a modal algebra commute with these symmetries in the following sense.

\begin{proposition} \label{prop:symmetries}
  If $\modframe{F}$ is a $\Delta_{\alpha}$-modal frame, then $\complex{(\modframe{F}^{-})} = (\complex{\modframe{F}})^{-}$ and $\complex{(\converse{\modframe{F}})} = \converse{\complex{\modframe{F}}}$. If $\alg{A}$ is a $\Delta_{\alpha}$-modal algebra, then $\canframe{(\alg{A}^{-})} = (\canframe{\alg{A}})^{-}$. If $\alg{A}$ is a tense $\Delta_{\alpha}$-modal algebra, then $\converse{\alg{A}^{-}} = (\converse{\alg{A}})^{-}$.
\end{proposition}

  This proposition in effect allows us to consider only signature contains $2^{3} - 2 = 6$ different sets of modal operators instead of $2^{4} - 4 = 12$ (excluding signatures with only one modal operator). It will be useful when discussing what we shall call the backward locality conditions in Section~\ref{sec:locality}.

\section{Unimodal frames and algebras}
\label{sec:unimodal}

  Having introduced modal frames with a single accessibility relation and corresponding modal algebras with a single modal operator or a single adjoint pair of modalities, we now turn our attention to multiple accessibility relations over the same poset and multiple modal operators over the same distributive lattice. If~no conditions connecting these modalities are postulated, such multimodal algebras and frames are nothing but tuples of $\Delta_{\alpha}$-modal algebras or frames over the same distributive lattice or poset. However, we shall be interested in the case where these modal operators capture different aspects of a \emph{single} modality. Such algebras and frames shall be called \emph{unimodal}. A \emph{unimodal signature} will be a set of modal symbols, and a \emph{tense unimodal signature} will be a set of modal symbols or adjoint modal symbols. The \emph{full (tense) unimodal signature} is the set of all modal symbols (or adjoint modal symbols). 

  Let $\lang{L}$ be a unimodal signature and let $\pair{W}{\leq}$ be a poset in the following. Recall that $\Delta_{\alpha}[R]$ denotes the least $\Delta_{\alpha}$-relation extending $R$, e.g.\ $\Box_{+}[R] = {\leq} \circ {R} \circ {\leq}$.

\begin{definition}
  An \emph{$\lang{L}$-modal frame} over $W$ is an $\lang{L}$-tuple $R_{\lang{L}}$ of monotone relations over $W$. An \emph{$\lang{L}$-modal algebra} over $\alg{L}$ is an expansion of $\alg{L}$ by a $\Delta_{\alpha}$-modal operator for each $\Delta_{\alpha} \in \lang{L}$.
\end{definition}

  An $\lang{L}$-modal frame over $W$ is to be viewed as a $\Delta_{\alpha}$-modal frame over $W$ for each $\Delta_{\alpha} \in \lang{L}$. Accordingly, the complex algebra of an $\lang{L}$-modal frame is an $\lang{L}$-modal algebra.

  We now wish to formalize the idea that the accessibility relations of an $\lang{L}$-modal frame in fact come from a single binary relation, just like they do in intuitionistic modal logic and positive modal logic.

\begin{definition}
  An \emph{$\lang{L}$-tuple $R_{\lang{L}}$ of monotone relations on $W$} is consists of a $\Delta_{\alpha}$-monotone relation $R^{\Delta}_{\alpha}$ on $W$ to each $\Delta_{\alpha} \in \lang{L}$. An $\lang{L}$-tuple $R_{\lang{L}}$ of monotone relations on $W$ is \emph{generated by the binary relation $R$} if $R^{\Delta}_{\alpha} = \Delta_{\alpha}[R]$ for each $\Delta_{\alpha} \in \lang{L}$. An \emph{$\lang{L}$-unimodal frame} over $W$ is an $\lang{L}$-modal frame over $W$ such that $R_{\lang{L}}$ is generated by some underlying relation $R$.
\end{definition}

  For each $\lang{L}$-tuple of monotone relations on $W$ generated by a single relation, without loss of generality we can always take the generating relation to be $R \assign \bigcap_{\Delta_{\alpha} \in \lang{L}} R^{\Delta}_{\alpha}$. It follows that $\lang{L}$-tuples of monotone relations generated by a single relation are in bijective correspondence with \emph{$\lang{L}$-convex relations} on the underlying set of $W$, i.e.\ binary relations $R$ such that $R = \bigcap_{\Delta_{\alpha} \in \lang{L}} \Delta_{\alpha}[R]$. It is therefore largely a matter of taste whether we choose to view $\lang{L}$-unimodal frames as posets equipped with $\lang{L}$-tuples of monotone relations or as posets equipped with an underlying $\lang{L}$-convex relation.

  The goal of this section is to axiomatize the universal class (or equivalently, the quasivariety) generated by the complex algebras of $\lang{L}$-unimodal frames, just as we axiomatized the universal class (or equivalently, the variety) generated by the complex algebras of $\Delta_{\alpha}$-frames in the previous section.

  It turns out that for this purpose it suffices to axiomatize the interaction between each \emph{pair} of modalities. We therefore need to consider $12$ different axioms. These are the axioms labelled $(\Delta_{\alpha}, \nabla_{\!\beta})$ shown in Figure~\ref{fig:unimodal-axioms}. Taking order duality into account, we only need to consider 6 conditions, which furthermore naturally split into 3 groups, as indicated in the figure.

\begin{figure}[t]
\caption{Modal axioms for the full unimodal signature}
\label{fig:unimodal-axioms}

\begin{align}
  \Box_{-} a \wedge \Box_{+} (a \vee b) \leq \Box_{+} b \tag{$\Box_{+}, \Box_{-}$}\\
  \Box_{+} a \wedge \Box_{-} (a \wedge b) \leq \Box_{-} b \tag{$\Box_{-}, \Box_{+}$}\\
  \Diamond_{+} b \leq \Diamond_{+} (a \wedge b) \vee \Diamond_{-} a \tag{$\Diamond_{+}, \Diamond_{-}$}\\
  \Diamond_{-} b \leq \Diamond_{-} (a \vee b) \vee \Diamond_{+} a \tag{$\Diamond_{-}, \Diamond_{+}$}
\end{align}
\begin{align}
  \Diamond_{-} a \wedge c \leq \Box_{+} a 	& \Rightarrow c \leq \Box_{+} a \tag{$\Box_{+}, \Diamond_{-}$}\\
  \Diamond_{+} a \wedge c \leq \Box_{-} a 	& \Rightarrow c \leq \Box_{-} a \tag{$\Box_{-}, \Diamond_{+}$}\\
  \Diamond_{+} a \leq \Box_{-} a \vee c 	& \Rightarrow \Diamond_{+} a \leq c \tag{$\Diamond_{+}, \Box_{-}$}\\
  \Diamond_{-} a \leq \Box_{+} a \vee c 	& \Rightarrow \Diamond_{-} a \leq c \tag{$\Diamond_{-}, \Box_{+}$}
\end{align}
\begin{align}
  \Diamond_{+} b \wedge c \leq \Box_{+} a  	& \Rightarrow \Box_{+} (a \vee b) \wedge c \leq \Box_{+} a \tag{$\Box_{+}, \Diamond_{+}$}\\
  \Diamond_{-} b \wedge c \leq \Box_{-} a 	& \Rightarrow \Box_{-} (a \wedge b) \wedge c \leq \Box_{-} a \tag{$\Box_{-}, \Diamond_{-}$}\\
  \Diamond_{+} a \leq \Box_{+} b \vee c 	& \Rightarrow \Diamond_{+} a \leq \Diamond_{+} (a \wedge b) \vee c \tag{$\Diamond_{+}, \Box_{+}$}\\
  \Diamond_{-} a \leq \Box_{-} b \vee c  	& \Rightarrow \Diamond_{-} a \leq \Diamond_{-} (a \vee b) \vee c \tag{$\Diamond_{-}, \Box_{-}$}
\end{align}

\end{figure}

\begin{figure}[t]
\caption{Equational formulations of the modal axioms for the full unimodal signature}
\label{fig:equational-unimodal-axioms}

\begin{align*}
  \Box_{-} a \wedge \Box_{+} (a \vee b) \leq \Box_{+} b \tag*{$(\Box_{+}, \Box_{-})$}\\
  \Box_{+} a \wedge \Box_{-} (a \wedge b) \leq \Box_{-} b \tag*{$(\Box_{-}, \Box_{+})$}\\
  \Diamond_{+} b \leq \Diamond_{+} (a \wedge b) \vee \Diamond_{-} a \tag*{$(\Diamond_{+}, \Diamond_{-})$}\\
  \Diamond_{-} b \leq \Diamond_{-} (a \vee b) \vee \Diamond_{+} a \tag*{$(\Diamond_{-}, \Diamond_{+})$}
\end{align*}
\begin{align*}
  \Diamond_{-} a \rightarrow \Box_{+} a \leq \Box_{+} a \tag*{$(\Box_{+}, \Diamond_{-})$}\\
  \Diamond_{+} a \rightarrow \Box_{-} a \leq \Box_{-} a \tag*{$(\Box_{-}, \Diamond_{+})$}\\
  \Diamond_{+} a \leq \Diamond_{+} a \minus \Box_{-} a \tag*{$(\Diamond_{+}, \Box_{-})$}\\
  \Diamond_{-} a \leq \Diamond_{-} a \minus \Box_{+} a \tag*{$(\Diamond_{-}, \Box_{+})$}
\end{align*}
\begin{align*}
  \Diamond_{+} b \rightarrow \Box_{+} a \leq \Box_{+} (a \vee b) \rightarrow \Box_{+} a \tag*{$(\Box_{+}, \Diamond_{+})$}\\
  \Diamond_{-} b \rightarrow \Box_{-} a \leq \Box_{-} (a \wedge b) \rightarrow \Box_{-} a \tag*{$(\Box_{-}, \Diamond_{-})$}\\
  \Diamond_{+} a \minus \Diamond_{+} (a \wedge b) \leq \Diamond_{+} a \minus \Box_{+} b\tag*{$(\Diamond_{+}, \Box_{+})$}\\
  \Diamond_{-} a \minus \Diamond_{-} (a \vee b) \leq \Diamond_{-} a \minus \Box_{-} b \tag*{$(\Diamond_{-}, \Box_{-})$}
\end{align*}

\end{figure}

  We achieve our goal in three steps. First, we determine the frame conditions which correspond to the quasiequations $(\Delta_{\alpha}, \nabla_{\!\beta})$. Secondly, we observe that although these frame conditions do not state that the modal accessibility relations are generated by a single underlying relation, they do so if we restrict to a certain wide class of frames, which in particular includes all canonical frames. Finally, we show that if an algebra satisfies the quasiequation $(\Delta_{\alpha}, \nabla_{\!\beta})$, then its canonical frame satisfies the corresponding frame condition. We obtain as a corollary that for most choices of $\lang{L}$ the class of $\lang{L}$-unimodal frames is not definable by means of universal sentences in the signature of $\lang{L}$-modal algebras (or even bi-Heyting $\lang{L}$-modal algebras).

\begin{definition}
  An \emph{$\lang{L}$-unimodal algebra} is an $\lang{L}$-modal algebra which satisfies the quasiequations $(\Delta_{\alpha}, \nabla_{\!\beta})$ for all $\Delta_{\alpha}, \nabla_{\!\beta} \in \lang{L}$.
\end{definition}

  All of the quasiequations $(\Delta_{\alpha}, \nabla_{\!\beta})$ may be expressed equationally in the presence of Heyting implication and co-implication, as shown in Figure~\ref{fig:equational-unimodal-axioms}. Bi-Heyting $\lang{L}$-unimodal algebras therefore in fact form a variety, as do Heyting $\lang{L}$-unimodal algebras for certain choices of $\lang{L}$.

\begin{proposition}
  The class of bi-Heyting $\lang{L}$-unimodal algebras is a variety for each $\lang{L}$. The class of Heyting $\lang{L}$-unimodal algebras is a variety if $\lang{L} = \{ \Box_{-}, \Box_{+} \}$ or $\lang{L} = \{ \Diamond_{-}, \Diamond_{+} \}$ or $\lang{L}$ contains only one modality.
\end{proposition}

  We now show that the quasiequations in Figure~\ref{fig:unimodal-axioms} correspond precisely to the frame conditions in Figure~\ref{fig:unimodal-frame-conditions}. Recall that the quasiequation $(\Delta_{\alpha}, \nabla_{\!\beta})$ \emph{corresponds to} the frame condition $(\Delta_{\alpha}, \nabla_{\!\beta})$ in case an $\lang{L}$-modal frame satisfies the frame condition $(\Delta_{\alpha}, \nabla_{\!\beta})$ if and only if its complex algebra satisfies the quasiequation~$(\Delta_{\alpha}, \nabla_{\!\beta})$. This of course presupposes that the modalities $\Delta_{\alpha}$ and $\nabla_{\!\beta}$ belong to the signature $\lang{L}$.

  In the following proofs, we use $a \in^{\alpha} A$ to abbreviate $a \in A$ for $\alpha = +$ and $a \notin A$ for $\alpha = -$. As expected, $\wedge^{+} = \wedge$ and $\wedge^{-} = \vee$, while $\vee^{+} = \vee$ and $\vee^{-} = \wedge$. If $\filter{V}$ is a prime filter on $\alg{A}$, then $\filter{V}^{+} = \filter{V}$ and $\filter{V}^{-} = \alg{A} \setminus \filter{V}$.

\begin{figure}[t]
\caption{Unimodal frame conditions}
\label{fig:unimodal-frame-conditions}

\begin{align*}
  R^{\Box}_{+} \subseteq (R^{\Box}_{+} \cap R^{\Box}_{-}) \circ {\leq} \tag*{$(\Box_{+}, \Box_{-})$}\\
  R^{\Box}_{-} \subseteq (R^{\Box}_{-} \cap R^{\Box}_{+}) \circ {\geq} \tag*{$(\Box_{-}, \Box_{+})$}\\
  R^{\Diamond}_{+} \subseteq (R^{\Diamond}_{+} \cap R^{\Diamond}_{-}) \circ {\geq} \tag*{$(\Diamond_{+}, \Diamond_{-})$}\\
  R^{\Diamond}_{-} \subseteq (R^{\Diamond}_{-} \cap R^{\Diamond}_{+}) \circ {\leq} \tag*{$(\Diamond_{-}, \Diamond_{+})$}
\end{align*}
\begin{align*}
  R^{\Box}_{+} \subseteq {\leq} \circ (R^{\Box}_{+} \cap R^{\Diamond}_{-}) \tag*{$(\Box_{+}, \Diamond_{-})$}\\
  R^{\Box}_{-} \subseteq {\leq} \circ (R^{\Box}_{-} \cap R^{\Diamond}_{+}) \tag*{$(\Box_{-}, \Diamond_{+})$}\\
  R^{\Diamond}_{+} \subseteq {\geq} \circ (R^{\Diamond}_{+} \cap R^{\Box}_{-}) \tag*{$(\Diamond_{+}, \Box_{-})$}\\
  R^{\Diamond}_{-} \subseteq {\geq} \circ (R^{\Diamond}_{-} \cap R^{\Box}_{+}) \tag*{$(\Diamond_{-}, \Box_{+})$}
\end{align*}
\begin{align*}
  \text{if } u R^{\Box}_{+} v, \text{ then there are } u' \geq u \text{ and } v' \leq v \text{ such that } u R^{\Box}_{+} v', u R^{\Box}_{+} v', \text{ and } u' R^{\Diamond}_{+} v' \tag*{$(\Box_{+}, \Diamond_{+})$}\\
  \text{if } u R^{\Box}_{-} v, \text{ then there are } u' \geq u \text{ and } v' \geq v \text{ such that } u R^{\Box}_{-} v', u R^{\Box}_{-} v', \text{ and } u' R^{\Diamond}_{-} v' \tag*{$(\Box_{-}, \Diamond_{-})$}\\
  \text{if } u R^{\Diamond}_{+} v, \text{ then there are } u' \leq u \text{ and } v' \geq v \text{ such that } u R^{\Diamond}_{+} v', u R^{\Diamond}_{+} v', \text{ and } u' R^{\Box}_{+} v' \tag*{$(\Diamond_{+}, \Box_{+})$}\\
  \text{if } u R^{\Diamond}_{-} v, \text{ then there are } u' \leq u \text{ and } v' \leq v \text{ such that } u R^{\Diamond}_{-} v', u R^{\Diamond}_{-} v', \text{ and } u' R^{\Box}_{-} v' \tag*{$(\Diamond_{-}, \Box_{-})$}
\end{align*}

\end{figure}

\begin{theorem}[Correspondence] \label{thm:correspondence}
  Each quasiequation $(\Delta_{\alpha},\! \nabla_{\!\beta})$ corresponds to the frame condition $(\Delta_{\alpha},\! \nabla_{\!\beta})$.
\end{theorem}

\begin{proof}
  In the left-to-right direction, we only verify the case of $(\Diamond_{\alpha}, \Box_{\alpha})$, hence by order duality also of~$(\Box_{\alpha}, \Diamond_{\alpha})$. The other left-to-right implications are easier and are therefore left to the interested reader as an exercise.

  Suppose that $u R^{\Diamond}_{\alpha} v$ but there are no $u' \leq u$ and $v' \geq^{\alpha} v$ such that $u' R^{\Diamond}_{\alpha} v$, $u R^{\Diamond}_{\alpha} v'$ and $u' R^{\Box}_{\alpha} v'$. Then let $w \in^{\alpha} a$ if and only if $w \geq^{\alpha} v$, let $w \in^{\alpha} b$ if and only if there is some $u' \leq u$ such that $u' R^{\Diamond}_{\alpha} v$ and $u' R^{\Box}_{\alpha} w$, and let $w \notin c$ if and only if $w \leq u$. It follows that $\Diamond_{\alpha} a \leq \Box_{\alpha} b \vee c$, but $u \in \Diamond_{\alpha} a$, $u \notin \Diamond_{+}(a \wedge^{\alpha} b)$ and $u \notin c$.

  To verify the right-to-left implications, by order duality it suffices to deal with the three cases $(\Box_{\alpha}, \Box_{-\alpha})$ and $(\Diamond_{\alpha}, \Box_{-\alpha})$ and $(\Diamond_{\alpha}, \Box_{\alpha})$. We deal with them in this order.

  Suppose that $u R^{\Box}_{\alpha} v$ but there is no $v^\prime \leq^{\alpha} v$ such that $u R^{\Box}_{\alpha} v^\prime$ and $u R^{\Box}_{-\alpha} v^\prime$. Then let $w \in^{-\alpha} a$ if and only if $u R^{\Box}_{-\alpha} w$ and let $w \in^{-\alpha} b$ if and only if $w \leq^{\alpha} v$. It follows that $u \in \Box_{-\alpha} a$ by definition and $u \in \Box_{\alpha} (a \vee^{\alpha} b)$ by the assumption about the frame, but $u \notin \Box_{\alpha} b$.

  Now suppose that $u R^{\Diamond}_{\alpha} v$ but there is no $u^\prime \leq u$ such that $u^\prime R^{\Diamond}_{\alpha} v$ and $u^\prime R^{\Box}_{-\alpha} v$. Then let $w \in^{\alpha} a$ if and only if $w \geq^{\alpha} v$, and let $w \notin c$ if and only if $w \leq u$. It follows that $\Diamond_{\alpha} a \leq \Box_{-\alpha} a \vee c$, but $u \in \Diamond_{\alpha} a$ and $u \notin c$.

  Finally, we deal with the case $(\Diamond_{\alpha}, \Box_{\alpha})$. Suppose that $\Diamond_{\alpha} a \leq \Box_{\alpha} b \vee c$ and $u R^{\Diamond}_{\alpha} v$ for some $u \notin c$ and $v \in a$. Then there are $u' \leq u$ and $v' \geq^{\alpha} v$ such that $u' R^{\Diamond}_{\alpha} v$, $u R^{\Diamond}_{\alpha} v'$ and $u' R^{\Box}_{\alpha} v'$, hence $u' \in \Diamond_{\alpha} a$. It follows that $u' \notin c$, therefore $u' \in \Box_{\alpha} b$, $v' \in^{\alpha} b$, and $u \in \Diamond_{\alpha} (a \wedge^{\alpha} b)$.
\end{proof}

  The first 8 conditions $(\Delta_{\alpha}, \nabla_{\!-\alpha})$ in Figure \ref{fig:unimodal-frame-conditions} state precisely that $R^{\Delta}_{\alpha} = \Delta_{\alpha}[R^{\Delta}_{\alpha} \cap R^{\nabla}_{-\alpha}]$, i.e.\ that the relations $R^{\Delta}_{\alpha}$ and $R^{\nabla}_{-\alpha}$ are generated by a single relation. By contrast, the last 4 conditions $(\Delta_{\alpha}, \nabla_{\!\alpha})$ state that for each $u R^{\Delta}_{\alpha} v$ there are $(u', v') \leq (u, v)$ in $(U, \leq) \times_{\Delta_{\alpha}} (U, \leq)$ such that $u' R^{\Delta}_{\alpha} v$, $u R^{\Delta}_{\alpha} v'$, and $u' R^{\nabla}_{\alpha} v'$. This condition is depicted on the left-hand side of Figure \ref{fig:condition-diamond-plus-box-plus} for $R^{\Diamond}_{+}$ and $R^{\Box}_{+}$ alongside the condition which we would have liked to capture, namely $R^{\Diamond}_{+}$ and $R^{\Box}_{+}$ being generated by a single relation.

  If the relations $R^{\Delta}_{\alpha}$ and $R^{\nabla}_{\beta}$ have the same tonicity in at least one of their two arguments, Theorem~\ref{thm:correspondence} tells us that we can characterize the class of frames such that $R^{\Delta}_{\alpha}$ and $R^{\nabla}_{\beta}$ are generated by a single underlying relation by means of a quasiequation in the signature of $\lang{L}$-modal algebras.

  The situation is more complicated when the relations $R^{\Delta}_{\alpha}$ and $R^{\nabla}_{\beta}$ have opposite tonicities in both positions. The following examples show that the conditions $(\Delta_{\alpha}, \nabla_{\beta})$ do not capture precisely the class of frames such that $R^{\Delta}_{\alpha}$ and $R^{\nabla}_{\beta}$ are generated by a single underlying relation, i.e.\ the class of frames such that $R^{\Diamond}_{\alpha} = \Diamond_{\alpha}[R^{\Diamond}_{\alpha} \cap R^{\Box}_{\alpha}]$ and $R^{\Box}_{\alpha} = \Box_{\alpha}[R^{\Box}_{\alpha} \cap R^{\Diamond}_{\alpha}]$.

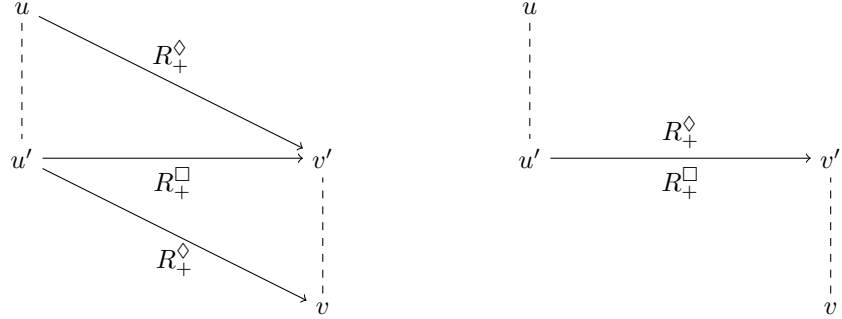
\begin{figure}[t]
\caption{The frame condition $(\Diamond_{+}, \Box_{+})$ vs.\ the property of being generated by a single relation}
\label{fig:condition-diamond-plus-box-plus}

\medskip

\begin{center}

\begin{tikzpicture}
\node (u) at (0,4) {$u$};
\node (v) at (4,0) {$v$};
\node (uprime) at (0,2) {$u'$};
\node (vprime) at (4,2) {$v'$};
\path[draw,dashed] (u) -- (uprime);
\path[draw,dashed] (v) -- (vprime);
\path[draw,->] (u) -- node[above] {$R^{\Diamond}_{+}$} (vprime);
\path[draw,->] (uprime) -- node[below] {$R^{\Diamond}_{+}$} (v);
\path[draw,->] (uprime) -- node[below] {$R^{\Box}_{+}$} (vprime);
\end{tikzpicture}
\qquad\qquad\qquad
\begin{tikzpicture}
\node (u) at (0,4) {$u$};
\node (v) at (4,0) {$v$};
\node (uprime) at (0,2) {$u'$};
\node (vprime) at (4,2) {$v'$};
\path[draw,dashed] (u) -- (uprime);
\path[draw,dashed] (v) -- (vprime);
\path[draw,->] (uprime) -- node[below] {$R^{\Box}_{+}$} node[above] {$R^{\Diamond}_{+}$} (vprime);
\end{tikzpicture}

\end{center}

\end{figure}

\begin{example} \label{ex:non-correspondence-box-diamond}
  Let $W = [0, 1] \times \{0, 1 \}$, let $(q, i) R^{\Box}_{+} (r, j)$ if and only if $i = 0$, $j = 1$ and $q \leq r$, and let $(q, i) R^{\Diamond}_{+} (r, j)$ if and only if $i = 0$, $j = 1$ and $q > r$. This frame satisfies $(\Diamond_{+}, \Box_{+})$ but $R^{\Box}_{+} \cap R^{\Diamond}_{+} = \emptyset$.
\end{example}

  This example, however, does not rule out the possibility that the conjunction of all the applicable frame conditions for some modal signature $\lang{L}$ does imply that $R^{\Diamond}_{+} = \Diamond_{+}[R^{\Box}_{+} \cap R^{\Diamond}_{+}]$. We rule this out by brute force.

\begin{example} \label{ex:non-correspondence-arbitrary-language}
  Let us call the last four pairs of modalities in Figure \ref{fig:unimodal-frame-conditions} (as well as the corresponding frame conditions) \emph{problematic}. We assume that $\lang{L}$ contains a problematic pair of modalities. Let $\modframe{F}_{0}$ be any $\lang{L}$-modal frame which at least contains some pair of points connected by some accessibility relation.

  We define $\modframe{F}_{2i+1}$ by adding enough points to $\modframe{F}_{2i}$ to make all of the problematic conditions $(\Delta_{\alpha}, \nabla_{\!\beta})$ hold for $\Delta_{\alpha}, \nabla_{\!\beta} \in \lang{L}$. That is, for any problematic pair of modalities $\Delta_{\alpha}$ and $\nabla_{\!\beta}$ in $\lang{L}$ and any pair of points $u$ and $v$ such that $u R^{\Delta}_{\alpha} v$, we add a pair of points $u^\prime$ and $v^\prime$ which ``complete'' the appropriate version of the \emph{left} part of Figure \ref{fig:condition-diamond-plus-box-plus}. (It would be straightforward but tedious to explicitly write out what this means. For instance in case $\Delta_{\alpha} = \Box_{+}$ and $\nabla_{\!\beta} = \Diamond_{+}$, we require that $u^\prime \leq w$ if and only if $w = u^\prime$ or $u \leq w$, that $u^\prime R^{\Diamond}_{+} w$ if and only if $w \leq v$, that $u^\prime R^{\Box}_{+} w$ if and only if $w = v^\prime$, and so on.) We define $\modframe{F}_{2i+2}$ by adding enough points to $\modframe{F}_{2i+1}$ to make all of the applicable \emph{un}problematic conditions $(\Delta_{\alpha}, \nabla_{\!\beta})$ hold. That is, for any unproblematic pair of modalities $\Delta_{\alpha}$ and $\nabla_{\!\beta}$ in $\lang{L}$ and any pair of points $u$ and $v$ such that $u R^{\Delta}_{\alpha} v$, we add a pair of points $u^\prime$ and $v^\prime$ such that $u^\prime (R^{\Delta}_{\alpha} \cap R^{\nabla}_{\beta}) v^\prime$ which ``complete'' the appropriate version of the \emph{right} part of Figure \ref{fig:condition-diamond-plus-box-plus}.

  We now define $\modframe{F}$ as the union of the sequence of frames $\modframe{F}_{i}$ for $i \in \omega$. The frame $\modframe{F}$ was constructed to satisfy all of the applicable conditions $(\Delta_{\alpha}, \nabla_{\!\beta})$. However, if $\Delta_{\alpha}$ and $\nabla_{\!\beta}$ is a problematic pair of modalities, then $u (R^{\Delta}_{\alpha} \cap R^{\nabla}_{\beta}) v$ holds in $\modframe{F}$ if and only if it holds already in the original frame $\modframe{F}_{0}$, as we never add any pair of points connected by $R^{\Delta}_{\alpha} \cap R^{\nabla}_{\beta}$. In particular, $u (R^{\Delta}_{\alpha} \cap R^{\nabla}_{\beta}) v$ never holds unless $u$ and $v$ are points in the original frame $\modframe{F}_{0}$. The problematic pair of relations $R^{\Delta}_{\alpha}$ and $R^{\nabla}_{\beta}$ is therefore \emph{never} generated by a single relation on the frame $\modframe{F}$, even though $\modframe{F}$ by construction satisfies $(\Delta_{\alpha}, \nabla_{\!\beta})$ for all $\Delta_{\alpha}$ and $\nabla_{\!\beta}$ in $\lang{L}$.
\end{example}

  Although the desired equivalence between the conditions in Figure~\ref{fig:unimodal-frame-conditions} and being a unimodal frame does not hold in full generality, it does hold on a wide class of frames which includes all canonical frames.

\begin{definition} \label{def:minimally-generated}
  Let $\pair{U}{\leq}$ and $\pair{V}{\sqleq}$ be posets. A set $X \subseteq U$ is \emph{minimally generated} if for each $u \in X$ the set $X$ has a minimal element below $u$. A set $X \subseteq U \times V$ is \emph{componentwise minimally generated} if for each $(u,v) \in X$, both of the sets $\set{w \in V}{uXw \textrm{ and } w \sqleq v}$ and $\set{w \in V}{wXv \textrm{ and } w \leq u}$ have a minimal element. A $\Delta_{\alpha}$-relation on $\pair{U}{\leq}$ is \emph{(componentwise) minimally ge\-ne\-ra\-ted} if it is (componentwise) minimally generated as a subset of $\pair{U}{\leq} \times_{\Delta_{\alpha}} \pair{U}{\leq}$. An $\lang{L}$-modal frame $\modframe{F}$ is (componentwise) minimally generated in case each $R^{\Delta}_{\alpha}$ in $\lang{L}$ is a (componentwise) minimally generated $\Delta_{\alpha}$-relation.
\end{definition}

  Each componentwise minimally generated relation is in particular minimally generated.

\begin{proposition} \label{prop:unimodal-minimally-generated}
  Let $\modframe{F}$ be a minimally generated $\lang{L}$-modal frame. Then $\modframe{F}$ is an $\lang{L}$-unimodal frame (i.e.\ $R^{\Delta}_{\alpha} = \Delta_{\alpha}[R^{\Delta}_{\alpha} \cap R^{\nabla}_{\beta}]$ for all $\Delta_{\alpha}, \nabla_{\beta} \in \lang{L}$) if and only if $\modframe{F}$ satisfies $(\Delta_{\alpha}, \nabla_{\!\beta})$ for each $\Delta_{\alpha}, \nabla_{\!\beta} \in \lang{L}$.
\end{proposition}

\begin{proof}
  We first show that $R^{\Delta}_{\alpha} = \Delta_{\alpha}[R^{\Delta}_{\alpha} \cap R^{\nabla}_{\beta}]$ and $(\Delta_{\alpha}, \nabla_{\!\beta})$ are equivalent conditions. For $(\Delta_{\alpha}, \nabla_{\!-\alpha})$ this holds by definition. Now let $\modframe{F}$ be a minimally generated frame which satisfies $(\Delta_{\alpha}, \nabla_{\!\alpha})$. By order duality we may assume without loss of generality that $\Delta = \Diamond$. Now if $u R^{\Diamond}_{\alpha} v$, then there are $u'$ and $v'$ in $\modframe{F}$ such that $u' \leq u$ and $v' \geq^{\alpha} v$ and whenever $u'' R^{\Diamond}_{+} v''$ for some $u'' \leq u$ and $v'' \geq^{\alpha} v'$, then in fact $u'' = u'$ and $v'' = v'$. Such worlds $u'$ and $v'$ will be said to constitute an \emph{$R^{\Diamond}_{\alpha}$-minimal pair}, and $R^{\Box}_{\alpha}$-minimal pairs are defined dually. Applying the condition $(\Diamond_{\alpha}, \Box_{\alpha})$ to $u'$ and $v'$ now yields that there are $u'' \leq u$ and $v'' \geq^{\alpha} v'$ such that $u' R^{\Diamond_{\alpha}} v''$ and $u'' R^{\Diamond}_{\alpha} v'$ and $u'' R^{\Box}_{\alpha} v''$. But then $u'' = u'$ and $v'' = v'$, hence $u' (R^{\Diamond}_{\alpha} \cap R^{\Box}_{\alpha}) v'$.

  Each $\lang{L}$-unimodal frame satisfies $R^{\Delta}_{\alpha} \subseteq \Delta_{\alpha}[R] \subseteq  \Delta_{\alpha}[R^{\Delta}_{\alpha} \cap R^{\nabla}_{\beta}]$ and $\Delta_{\alpha}[R^{\Delta}_{\alpha} \cap R^{\nabla}_{\beta}] \subseteq \Delta_{\alpha}[R^{\Delta_{\alpha}}] \subseteq R^{\Delta}_{\alpha}$. Conversely, suppose that $\modframe{F}$ is a minimally generated $\lang{L}$-frame which satisfies $R^{\Delta}_{\alpha} = \Delta_{\alpha}[R^{\Delta}_{\alpha} \cap R^{\nabla}_{\beta}]$ for each $\Delta_{\alpha}, \nabla_{\!\beta} \in \lang{L}$. It suffices to prove that if $u$ and $v$ form a $R^{\Delta}_{\alpha}$-minimal pair, then $u R^{\nabla}_{\beta} v$ for each $\nabla_{\!\beta} \in \lang{L}$. But this holds because by assumption $R^{\Delta}_{\alpha} = \Delta_{\alpha}[R^{\Delta}_{\alpha} \cap R^{\nabla}_{\beta}]$, hence there are $u' \leq u$ and $v' \geq^{\alpha} v$ in case $\Delta = \Diamond$ and $u' \geq u$ and $v' \leq^{\alpha} v$ in case $\Delta = \Box$ such that $u' R^{\Delta}_{\alpha} v'$ and $u' R^{\nabla}_{\beta} v'$. But by the $R^{\Delta}_{\alpha}$-minimality of $u$ and $v$ we have $u' = u$ and $v' = v$, therefore $u R^{\nabla}_{\beta} v$ and $u R v$ for each $R^{\Delta}_{\alpha}$-minimal pair $u$ and $v$.
\end{proof}

\begin{proposition} \label{prop:canonical-minimally-generated}
  The canonical frame of any $\lang{L}$-modal algebra is (componentwise) minimally generated.
\end{proposition}

\begin{proof}
  Let $\alg{A}$ be an $\lang{L}$-modal algebra. It suffices to consider singleton $\lang{L}$. We only deal with the cases $\lang{L} \subseteq \{ \Diamond_{+}, \Diamond_{-} \}$, since the cases $\lang{L} \subseteq \{ \Box_{+}, \Box_{-} \}$ are order dual. 

  If $\lang{L} = \{ \Diamond_{+} \}$, let $\filter{U}$ and $\filter{V}$ be prime filters on $\alg{A}$, and suppose that $\filter{U} R^{\Diamond}_{+} \filter{V}$, that is, $\Diamond_{+} [\filter{V}] \subseteq \filter{U}$. Use Zorn's lemma to extend $\filter{V}$ to a maximal filter $\filter{V}^\prime$ such that $\Diamond_{+} [\filter{V}^\prime] \subseteq \filter{U}$. It suffices to show that $\filter{V}^\prime$ is prime. If $a \notin \filter{V}^\prime$ and $b \notin \filter{V}^\prime$, then there is some $c \in \filter{V}^\prime$ such that $\Diamond_{+} (a \wedge c) \nsubseteq \filter{U}$ and some $d \in \filter{V}^\prime$ such that $\Diamond_{+} (b \wedge d) \nsubseteq \filter{U}$. Without loss of generality, we may assume that $c = d$. Since $\filter{U}$ is prime, $\Diamond_{+} (a \wedge c) \vee \Diamond_{+} (b \wedge c) = \Diamond_{+} ((a \wedge c) \vee (b \wedge c)) = \Diamond_{+} ((a \vee b) \wedge c) \nsubseteq \filter{U}$, hence $a \vee b \notin \filter{V}^\prime$.

  If $\lang{L} = \{ \Diamond_{-} \}$, suppose that $\filter{U} R^{\Diamond}_{-} \filter{V}$, that is, $\Diamond_{-} [\alg{A} \setminus \filter{V}] \subseteq \filter{U}$. Again, use Zorn's lemma to extend $\alg{A} \setminus \filter{V}$ to a maximal ideal $\ideal{W}^\prime$ such that $\Diamond_{-} [\ideal{W}^\prime] \subseteq \filter{U}$. It suffices to show that $\ideal{W}^\prime$ is prime. If $a \notin \filter{W}^\prime$ and $b \notin \filter{W}^\prime$, then, as above, there is some $c \in \ideal{W}^\prime$ such that $\Diamond_{-} (a \vee c) \nsubseteq \filter{U}$ and $\Diamond_{-}(b \vee c) \nsubseteq \filter{U}$, hence $\Diamond_{-}((a \wedge b) \vee c) \nsubseteq \filter{U}$ and $a \wedge b \notin \ideal{W}^\prime$.
\end{proof}

  To show that the quasiequations $(\Delta_{\alpha}, \nabla_{\!\beta})$ axiomatize the universal class generated by $\lang{L}$-unimodal frames, it only remains to prove they are canonical. We do this by showing that the canonical frame of each $\lang{L}$-unimodal algebra is an $\lang{L}$-unimodal frame for every unimodal signature $\lang{L}$.

\begin{theorem}[Canonicity] \label{thm:canonicity}
  The quasivariety of $\lang{L}$-unimodal algebras is canonical for each $\lang{L}$.
\end{theorem}

\begin{proof}
  Let $\alg{A}$ be an $\lang{L}$-unimodal algebra which satisfies the quasiequation $(\Delta_{\alpha}, \nabla_{\!\-\alpha})$. It suffices to show that $R^{\Delta}_{\alpha} = \Delta_{\alpha}[R^{\Delta}_{\alpha} \cap R^{\nabla}_{-\alpha}]$ holds in the canonical frame of $\alg{A}$. Appealing to order duality, we may assume without loss of generality that $\nabla = \Box$.

  In case $\Delta = \Box$, suppose that $\filter{U} R^{\Box}_{\alpha} \filter{V}$, that is, $\Box_{\alpha} [\alg{A} \setminus \filter{V}^{\alpha}] \subseteq \alg{A} \setminus \filter{U}$. By Proposition \ref{prop:canonical-minimally-generated}, we may assume that $\filter{V}$ is a minimal$^{\alpha}$ such prime filter. It suffices to show that $\filter{U} R^{\Box}_{-\alpha} \filter{V}$. Suppose therefore that $a \in^{\alpha} \filter{V}$. Then there is some $b \notin \filter{V}^{\alpha}$ such that $\bigcup \Box_{\alpha} (a \vee^{\alpha} b) \nsubseteq \alg{A} \setminus \filter{U}$. If it were the case that $\Box_{-\alpha} a \nsubseteq \alg{A} \setminus \filter{U}$, then we would have $\Box_{\alpha} b \nsubseteq \alg{A} \setminus \filter{U}$, which contradicts $b \notin \filter{V}^{\alpha}$.

  In case $\Delta = \Diamond$, suppose that $\filter{U} R^{\Diamond}_{\alpha} \filter{V}$, that is, $\Diamond_{\alpha} [\filter{V}^{\alpha}] \subseteq \filter{U}$. By Proposition \ref{prop:canonical-minimally-generated}, we may assume that $\filter{U}$ is a minimal such prime filter. It suffices to show that $\filter{U} R^{\Box}_{-\alpha} \filter{V}$. Suppose therefore that $\Box_{-\alpha} a \nsubseteq \alg{A} \setminus \filter{U}$. Then there are some $b \in \filter{V}^{\alpha}$ and $c \notin \filter{U}$ such that $\Diamond_{\alpha} b \leq \Box_{\alpha} a \vee c$, hence $\Diamond_{\alpha} (a \wedge^{\alpha} b) \leq \Box_{-\alpha} (a \wedge b) \vee c$ and $\Diamond_{\alpha} (a \wedge^{\alpha} b) \leq c$. If $a \in \filter{V}^{\alpha}$, then $a \wedge^{\alpha} b \in \filter{V}^{\alpha}$ and $\Diamond_{\alpha} (a \wedge^{\alpha} b) \in \filter{U}$, contradicting $c \notin \filter{U}$.

  The rest of the proof deals with the case $\Delta = \Diamond$ and $\alpha = +$. The case $\Delta = \Diamond$ and $\alpha = -$ is entirely analogous and it would be possibly to do both proofs at the same time as in the case $\Delta = \Box$, but we prefer not to deal with both cases at the same time for the sake of easier readability.

  Suppose that $\filter{U} R^{\Diamond}_{+} \filter{V}$, that is, $\Diamond_{+} [\filter{V}] \subseteq \filter{U}$. By Proposition \ref{prop:canonical-minimally-generated}, we may assume that $\filter{V}$ is a maximal such prime filter. It suffices to find a prime filter $\filter{W} \subseteq \filter{U}$ such that $\Diamond_{+} [\filter{V}] \subseteq \filter{W}$ and $\Box_{+} [\alg{A} \setminus \filter{V}] \subseteq \alg{A} \setminus \filter{W}$.

  To construct $\filter{W}$, we need to introduce some basic constructions on filters. Given a pair of filters $\filter{F}$ and $\filter{G}$, let $\filter{F} \vee \filter{G}$ be the filter generated by $\filter{F} \cup \filter{G}$. Given a filter $\filter{F}$ on $\alg{A}$, let $\Box_{+}^{-1} [\filter{F}]$ be the filter $\set{a \in \alg{A}}{\Box_{+} a \in \filter{F}}$. Finally, given an ideal $\ideal{I}$ and a filter $\filter{F}$ on $\alg{A}$, let $\filter{F} \minus \ideal{I}$ be the filter $\set{a \in \alg{A}}{f \leq a \vee i \textrm{ for some $f \in \filter{F}$, $i \in \ideal{I}$}}$. We say that a filter $\filter{F}$ \emph{knows that it is below} $\filter{U}$ if $\filter{F} = \filter{F} \minus (\alg{A} \setminus \filter{U})$, i.e.\ if $f \leq a \vee i$ for $f \in \filter{F}$ and $i \notin \filter{U}$ implies $a \in \filter{F}$. In particular $\filter{F} \subseteq \filter{U}$ whenever a proper filter $\filter{F}$ knows that it is below $\filter{U}$.

  Now let $\filter{W} = \Diamond_{+} [\filter{V}] \minus (\alg{A} \setminus \filter{U})$. Since $\Diamond_{+} [\filter{V}] \subseteq \filter{U}$, the filter $\filter{W}$ is proper. Furthermore, the filter $\filter{W}$ knows that it is below $\filter{U}$ since $\Diamond_{+} [\filter{V}] \minus (\alg{A} \setminus \filter{U}) = (\Diamond_{+} [\filter{V}] \minus (\alg{A} \setminus \filter{U})) \minus (\alg{A} \setminus \filter{U})$. Let $\filter{V}^\prime = \filter{V} \vee \Box_{+}^{-1} [\filter{W}]$, i.e.\ $v^\prime \in \filter{V}^\prime$ if and only if $v^\prime = v_1 \wedge a$ for some $v_1 \in \filter{V}$ and some $a \in \alg{A}$ such that $\Diamond_{+} v_2 \leq \Box_{+} a \vee i$ for some $v_2 \in \filter{V}$ and $i \notin \filter{U}$. Let $v = v_1 \wedge v_2$. Then $v^{\prime} \geq v \wedge a$ and $\Diamond_{+} v \leq \Box_{+} a \vee i$, hence $\Diamond_{+} v \leq \Diamond_{+} (v \wedge a) \vee i$. It follows that $\Diamond_{+} (v \wedge a) \in \filter{U}^\prime$. In other words, $\Diamond_{+} [\filter{V}^\prime] \subseteq \filter{W} \subseteq \filter{U}$. By the maximality of $\filter{V}$ it follows that in fact $\filter{V}^\prime = \filter{V}$.

  We therefore have a filter $\filter{W}$ which knows it is below $\filter{U}$ such that $\Box_{+}^{-1} [\filter{W}] \subseteq \filter{V}$. We extend $\filter{W}$ to a filter $\filter{U}^\prime$ which is maximal among filters which know that they are below $\filter{U}$. Clearly $\filter{U}^\prime R^{\Box}_{+} \filter{V}^\prime$ and $\filter{U}^\prime \subseteq \filter{U}$. Moreover, $\filter{U}^\prime R^{\Diamond}_{+} \filter{V}^\prime$, since $\filter{V}' = \filter{V}$. It therefore suffices to show that the filter $\filter{U}^\prime$ is prime.

  If $a \notin \filter{U}^\prime$, then there are some $c \notin \filter{V}^\prime$ and $i \notin \filter{U}$ such that $a \leq \Box_{+} c \vee i$. Likewise, if $b \notin \filter{U}^\prime$, then there are some $d \notin \filter{V}^\prime$ and $j \notin \filter{U}$ such that $b \leq \Box_{+} d \vee j$. Then $a \vee b \leq \Box_{+} c \vee \Box_{+} d \vee i \vee j \leq \Box_{+} (c \vee d) \vee (i \vee j)$. But since $\filter{U}$ and $\filter{V}^\prime$ are prime, $c \vee d \notin \filter{V}^\prime$ and $i \vee j \notin \filter{U}$, hence $a \vee b \notin \filter{U}^\prime$.
\end{proof}

\begin{corollary} \label{cor:canonical-frame-is-unimodal}
  The canonical frame of each $\lang{L}$-unimodal algebra is an $\lang{L}$-unimodal frame.
\end{corollary}

  As an immediate corollary we obtain a Kripke completeness theorem, i.e.\ each quasiequation (in fact, each universal sentence) which fails in some $\lang{L}$-unimodal algebra must fail in some $\lang{L}$-unimodal frame.

\begin{theorem}[Kripke completeness] \label{thm:completeness}
  The quasivariety of (Heyting, bi-Heyting) $\lang{L}$-unimodal algebras is generated as a universal class by the class of all complex algebras of $\lang{L}$-unimodal frames.
\end{theorem}

  Let us also consider the definability of the class of $\lang{L}$-unimodal frames by equational or quasiequational axioms. We~say that a class $\class{K}$ of $\lang{L}$-modal frames is \emph{definable} by a set of sentences if for each $\lang{L}$-modal frame we have $\modframe{F} \in \class{K}$ if and only if the complex algebra $\modframe{F}^{+}$ satisfies all of these sentences. Let $\lang{L}_{\class{DLat}}$ and $\lang{L}_{\class{BiHA}}$ be the signatures of distributive lattices and bi-Heyting algebras, respectively.

\begin{theorem}[Definability] \label{thm:definability}
  Let $\class{Fr}_{\lang{L}}$ be the class of all $\lang{L}$-unimodal frames. The following are equivalent:
\begin{itemize}
\item[(i)] $\class{Fr}_{\lang{L}}$ is definable by quasiequations in $\lang{L}_{\class{DLat}} \cup \lang{L}$.
\item[(ii)] $\class{Fr}_{\lang{L}}$ is definable by universal sentences in $\lang{L}_{\class{DLat}} \cup \lang{L}$.
\item[(iii)] $\class{Fr}_{\lang{L}}$ is definable by equations in $\lang{L}_{\class{BiHA}} \cup \lang{L}$.
\item[(iv)] $\class{Fr}_{\lang{L}}$ is definable by universal sentences in $\lang{L}_{\class{BiHA}} \cup \lang{L}$.
\item[(v)]$\lang{L}$ contains at most two modal operators and $\lang{L} \neq \{ \Box_{+}, \Diamond_{+} \}$ and $\lang{L} \neq \{ \Box_{-}, \Diamond_{-} \}$.
\end{itemize}
\end{theorem}

\begin{proof}
  The implications (v) $\Rightarrow$ (i) and (v) $\Rightarrow$ (iii) follow from Theorem \ref{thm:correspondence}. The implications (iii) $\Rightarrow$ (iv) and (i) $\Rightarrow$ (ii) $\Rightarrow$ (iv) are trivial. It remains to prove the implication (iv) $\Rightarrow$ (v). If $\class{Fr}_{\lang{L}}$ is definable by some set of universal sentences in $\lang{L}_{\class{BiHA}} \cup \lang{L}$, then by Theorem \ref{thm:completeness} these universal sentences axiomatize the class of all $\lang{L}$-unimodal algebras. They are therefore equivalent to the quasiequations which we used to define the class of all $\lang{L}$-unimodal algebras. But Example \ref{ex:non-correspondence-arbitrary-language} shows that if $\lang{L}$ contains a problematic pair of modalities, i.e.\ if it does not satisfy the assumption of (v), then the class of $\lang{L}$-unimodal frames is not defined by these quasiequations.
\end{proof}

  We therefore cannot in general view the unirelational semantics (given in terms of $R$) as a special case of the multirelational semantics (given in terms of the relations $R^{\Delta}_{\alpha}$) defined by restricting to a class of frames which satisfies some additional universal axioms. On the other hand, by Propositions \ref{prop:unimodal-minimally-generated} and \ref{prop:canonical-minimally-generated} it is true that relative to the class of minimally generated $\lang{L}$-modal frames (in particular, relative to the class of canonical $\lang{L}$-modal frames) unimodality may be expressed by an equation in $\lang{L}_{\class{BiHA}} \cup \lang{L}$ or by a quasiequation in $\lang{L}_{\class{DLat}} \cup \lang{L}$ for each unimodal signature~$\lang{L}$.

\section{Modal locality conditions}
\label{sec:locality}

  In the previous section, we axiomatized the universal class generated by the complex algebras of unimodal frames. In the current section, we show how to axiomatize, in a modular way, the universal classes generated by the complex algebras of unimodal frames satisfying certain natural compability conditions connecting the accessibility relation~$R$ and the partial order~$\leq$. This will in particular subsume intuitionistic modal logic and positive modal logic as special cases.

  Throughout this section we assume that $\modframe{F}$ is an \emph{$\lang{L}$-unimodal} frame over the poset $\pair{W}{\leq}$ generated by an \emph{$\lang{L}$-convex} relation $R$. We shall be interested in the following four \emph{forward locality conditions}:
\begin{align}
  {\geq} \circ {R} \subseteq {R} \circ {\geq}, \tag*{$(\Diamond_{+})^{*}$}\\
  {\geq} \circ {R} \subseteq {R} \circ {\leq}, \tag*{$(\Diamond_{-})^{*}$}\\
  {\leq} \circ {R} \subseteq {R} \circ {\leq}, \tag*{$(\Box_{+})^{*}$}\\
  {\leq} \circ {R} \subseteq {R} \circ {\geq}, \tag*{$(\Box_{-})^{*}$}
\end{align}
as well as the following four \emph{backward locality conditions}:
\begin{align}
  {R} \circ {\leq} \subseteq {\leq} \circ {R}, \tag*{$(\AdjDiamond_{+})^{*}$}\\
  {R} \circ {\leq} \subseteq {\geq} \circ {R}, \tag*{$(\AdjDiamond_{-})^{*}$}\\
  {R} \circ {\geq} \subseteq {\geq} \circ {R}, \tag*{$(\AdjBox_{+})^{*}$}\\
  {R} \circ {\geq} \subseteq {\leq} \circ {R}. \tag*{$(\AdjBox_{-})^{*}$}
\end{align}
The backward conditions are obtained from the forward ones by applying forward--backward duality, i.e.\ by replacing $R$ by its converse. The conjunction of $(\Diamond_{+})^{*}$ and $(\Box_{+})^{*}$ defines the semantics of positive modal logic, while the conjunction of $(\Diamond_{+})^{*}$ and $(\AdjDiamond_{+})^{*}$ defines the semantics of intuitionistic modal logic.

  We first verify that these conditions deserve to be called locality conditions: they hold if and only if the appropriate modal operator is evaluated locally with respect to the partial order. For example, recall that $\Diamond_{+} a$ holds at a world $u$ (i.e.\ $u \in \Diamond_{+} a$) if and only if there are worlds $v$ and $w$ such that $u \geq v$ and $v R w$ and $a$ holds at $w$ (i.e.\ $w \in a$). By \emph{local evaluation} we mean cutting out the middleman and replacing the above condition by a simpler one: $\Diamond_{+} a$ holds at $u$ if and only if there is a world $w$ where $a$ holds such that $u R w$.

  We stick to the conventions used in the previous section: $v \in^{\alpha} a$ stands for $v \in a$ if $\alpha = +$ and for $v \notin a$ if $\alpha = -$, and $\leq^{\alpha}$ ($\geq^{\alpha}$) stands for $\leq$ ($\geq$) if $\alpha = +$ and for $\geq$ ($\leq$) if $\alpha = -$.

\begin{proposition} \label{prop:local-modalities}
  The frame $\modframe{F}$ satisfies $(\Diamond_{\alpha})^{*}$ if and only if for each $a \in \modframe{F}^{+}$
\begin{align*}
  \Diamond_{\alpha} a = \set{u \in W}{u R v \text{ for some }{ v \in^{\alpha} a}}
\end{align*}
  It satisfies $(\Box_{\alpha})^{*}$ if and only if for each $a \in \modframe{F}^{+}$
\begin{align*}
  \Box_{\alpha} a = \set{u \in W}{u R v \text{ implies } v \in^{\alpha} a}.
\end{align*}
\end{proposition}

\begin{proof}
  By order reversal it suffices to prove the claim for $\Diamond_{\alpha}$. The left-to-right direction is easy, we only prove the right-to-left direction. Suppose therefore that $u \geq u^\prime R v^\prime$ but there is no $v \in W$ such that $u R v$ and $v \geq^{\alpha} v^\prime$. Then let $w \in^{\alpha} a$ if and only if $w \geq^{\alpha} v^\prime$. It follows that $u \in \Diamond_{\alpha} a$ but there is no $v \in W$ such that $u R v$ and $v \in^{\alpha} a$.
\end{proof}

\begin{proposition} \label{prop:backward-local-modalities}
  The frame $\modframe{F}$ satisfies $(\AdjDiamond_{\alpha})^{*}$ if and only if for each $a \in \modframe{F}^{+}$
\begin{align*}
  \AdjDiamond_{\alpha} a = \set{u \in W}{v R u \text{ for some }{ v \in^{\alpha} a}}.
\end{align*}
  It satisfies $(\AdjBox_{\alpha})^{*}$ if and only if for each $a \in \modframe{F}^{+}$
\begin{align*}
  \AdjBox_{\alpha} a = \set{u \in W}{v R u \text{ implies } v \in^{\alpha} a}.
\end{align*}
\end{proposition}

\begin{proof}
  The claim follows from Proposition \ref{prop:local-modalities} by forward--backward duality.
\end{proof}

  Each of the frame conditions $(\Delta_{\alpha})^{*}$ may in fact be decomposed as a conjunction of two simpler conditions. (Recall that we are assuming that $R = \bigcap_{\Delta_{\alpha} \in \lang{L}} R^{\Delta}_{\alpha}$.) For the forward conditions we get the following pairs:
\begin{align}
  R^{\Diamond}_{+} & \subseteq (R^{\Diamond}_{+} \cap R^{\Box}_{+}) \circ {\geq}, & R^{\Diamond}_{+} & \subseteq R^{\Box}_{-}, \tag*{$(\Diamond_{+}, \Box_{+})^{*} ~ \& ~ (\Diamond_{+}, \Box_{-})^{*}$} \\
  R^{\Diamond}_{-} & \subseteq (R^{\Diamond}_{-} \cap R^{\Box}_{-}) \circ {\leq}, & R^{\Diamond}_{-} & \subseteq R^{\Box}_{+}, \tag*{$(\Diamond_{-}, \Box_{-})^{*} ~ \& ~ (\Diamond_{-}, \Box_{+})^{*}$} \\
  R^{\Box}_{+} & \subseteq (R^{\Box}_{+} \cap R^{\Diamond}_{+}) \circ {\leq}, & R^{\Box}_{+} & \subseteq R^{\Diamond}_{-}, \tag*{$(\Box_{+}, \Diamond_{+})^{*} ~ \& ~ (\Box_{+}, \Diamond_{-})^{*}$} \\
  R^{\Box}_{-} & \subseteq (R^{\Box}_{-} \cap R^{\Diamond}_{-}) \circ {\geq}, & R^{\Box}_{-} & \subseteq R^{\Diamond}_{+},  \tag*{$(\Box_{-}, \Diamond_{-})^{*} ~ \& ~ (\Box_{-}, \Diamond_{+})^{*}$}
\end{align}
  while for the backward conditions we get the following pairs:
\begin{align}
  R^{\Box}_{+} & \subseteq {\leq} \circ (R^{\Box}_{+} \cap R^{\Diamond}_{+}), & R^{\Box}_{+} & \subseteq R^{\Box}_{-}, \tag*{$(\AdjDiamond_{+}, \AdjBox_{+})^{*} ~ \& ~ (\AdjDiamond_{+}, \AdjBox_{-})^{*}$} \\
  R^{\Diamond}_{-} & \subseteq {\geq} \circ (R^{\Diamond}_{-} \cap R^{\Box}_{-}), & R^{\Diamond}_{-} & \subseteq R^{\Diamond}_{+}, \tag*{$(\AdjDiamond_{-}, \AdjBox_{-})^{*} ~ \& ~ (\AdjDiamond_{-}, \AdjBox_{+})^{*}$} \\
  R^{\Diamond}_{+} & \subseteq {\geq} \circ (R^{\Diamond}_{+} \cap R^{\Box}_{+}), & R^{\Diamond}_{+} & \subseteq R^{\Diamond}_{-}, \tag*{$(\AdjBox_{+}, \AdjDiamond_{+})^{*} ~ \& ~ (\AdjBox_{+}, \AdjDiamond_{-})^{*}$} \\
  R^{\Box}_{-} & \subseteq {\leq} \circ (R^{\Box}_{-} \cap R^{\Diamond}_{-}), & R^{\Box}_{-} & \subseteq R^{\Box}_{+}. \tag*{$(\AdjBox_{-}, \AdjDiamond_{-})^{*} ~ \& ~ (\AdjBox_{-}, \AdjDiamond_{+})^{*}$}
\end{align}
  The next proposition states this claim more precisely. Each of its four equivalences may be derived from any of the others by applying order duality or forward--backward duality or both.

\begin{proposition} \label{prop:analysis}
  Let $\modframe{F}$ be an $\lang{L}$-unimodal frame. If $\Diamond_{\alpha} \in \lang{L}$, then $\modframe{F}$ satisfies $(\Diamond_{\alpha})^{*}$ if and only if
\begin{align*}
  \modframe{F} \text{ satisfies } (\Diamond_{\alpha}, \Box_{\alpha})^{*} \text{ in case } \Box_{\alpha} \in \alg{L} \text{ and } \modframe{F} \text{ satisfies } (\Diamond_{\alpha}, \Box_{-\alpha})^{*} \text{ in case } \Box_{-\alpha} \in \lang{L}.
\end{align*}
  If $\Box_{\alpha} \in \lang{L}$, then $\modframe{F}$ satisfies $(\Box_{\alpha})^{*}$ if and only if
\begin{align*}
  \modframe{F} \text{ satisfies } (\Box_{\alpha}, \Diamond_{\alpha})^{*} \text{ in case } \Diamond_{\alpha} \in \lang{L} \text{ and } \modframe{F} \text{ satisfies } (\Box_{\alpha}, \Diamond_{-\alpha})^{*} \text{ in case } \Diamond_{-\alpha} \in \lang{L}.
\end{align*}
  Let $\converse{\lang{L}}$ be the set of modalities adjoint to those in $\lang{L}$. If $\AdjDiamond_{\alpha} \in \converse{\lang{L}}$, then $\modframe{F}$ satisfies $(\AdjDiamond_{\alpha})^{*}$ if and only if
\begin{align*}
  \modframe{F} \text{ satisfies } (\AdjDiamond_{\alpha}, \AdjBox_{\alpha})^{*} \text{ in case } \AdjBox_{\alpha} \in \lang{L} \text{ and } \modframe{F} \text{ satisfies } (\AdjDiamond_{\alpha}, \AdjBox_{-\alpha})^{*} \text{ in case } \AdjBox_{-\alpha} \in \lang{L}.
\end{align*}
  If $\AdjBox_{\alpha} \in \converse{\lang{L}}$, then $\modframe{F}$ satisfies $(\AdjBox_{\alpha})^{*}$ if and only if
\begin{align*}
  \modframe{F} \text{ satisfies } (\AdjBox_{\alpha}, \AdjDiamond_{\alpha})^{*} \text{ in case } \AdjDiamond_{\alpha} \in \lang{L} \text{ and } \modframe{F} \text{ satisfies } (\AdjBox_{\alpha}, \AdjDiamond_{-\alpha})^{*} \text{ in case } \AdjDiamond_{-\alpha} \in \lang{L}.
\end{align*}
\end{proposition}

\begin{proof}
  We omit the easy proof of this proposition.
\end{proof}

  We now show that each of these 16 conditions corresponds to a canonical equation. For the 8 forward conditions these are equations in the signature of modal algebras. However, for the 8 backward conditions we have to add Heyting implication or co-implication in order to capture them algebraically.

\begin{figure}[t]
\caption{Forward locality axioms}
\label{fig:forward-locality-axioms}

\begin{align}
&  \Box_{+} a \wedge \Diamond_{+} b \leq \Diamond_{+} (a \wedge b) \tag*{$(\Diamond_{+}, \Box_{+})^{*}$}\\
&  \Box_{-} a \wedge \Diamond_{-} b \leq \Diamond_{-} (a \vee b) \tag*{$(\Diamond_{-}, \Box_{-})^{*}$}\\
&  \Box_{+} (a \vee b) \leq \Box_{+} a \vee \Diamond_{+} b \tag*{$(\Box_{+}, \Diamond_{+})^{*}$}\\
&  \Box_{-} (a \wedge b) \leq \Box_{-} a \vee \Diamond_{-} b \tag*{$(\Box_{-}, \Diamond_{-})^{*}$}
\end{align}
\begin{align}
&  \Box_{-} a \wedge \Diamond_{+} a \leq \bot \tag*{$(\Diamond_{+}, \Box_{-})^{*}$}\\
&  \Box_{+} a \wedge \Diamond_{-} a \leq \bot \tag*{$(\Diamond_{-}, \Box_{+})^{*}$}\\
&  \top \leq \Box_{+} a \vee \Diamond_{-} a \tag*{$(\Box_{+}, \Diamond_{-})^{*}$}\\
&  \top \leq \Box_{-} a \vee \Diamond_{+} a \tag*{$(\Box_{-}, \Diamond_{+})^{*}$}
\end{align}

\end{figure}

\begin{figure}[t]
\caption{Backward locality axioms}
\label{fig:backward-locality-axioms}

\begin{align*}
& \Diamond_{+} a \rightarrow \Box_{+} b \leq \Box_{+} (a \rightarrow b) \tag*{$(\AdjDiamond_{+}, \AdjBox_{+})^{*}$} \\
&  \Diamond_{-} (a \rightarrow b) \leq \Diamond_{-} b \minus \Box_{-} a \tag*{$(\AdjDiamond_{-}, \AdjBox_{-})^{*}$} \\
& \Diamond_{+} (b \minus a) \leq \Diamond_{+} b \minus \Box_{+} a \tag*{$(\AdjBox_{+}, \AdjDiamond_{+})^{*}$} \\
&  \Diamond_{-} a \rightarrow \Box_{-} b \leq \Box_{-} (b \minus a) \tag*{$(\AdjBox_{-}, \AdjDiamond_{-})^{*}$}
\end{align*}
\begin{align*}
&  \Box_{-} a \leq \Box_{+} (a \rightarrow \bot) \tag*{$(\AdjDiamond_{+}, \AdjBox_{-})^{*}$} \\
&  \Diamond_{-} (a \rightarrow \bot) \leq \Diamond_{+} a \tag*{$(\AdjDiamond_{-}, \AdjBox_{+})^{*}$} \\
&  \Diamond_{+} (\top \minus a) \leq \Diamond_{-} a \tag*{$(\AdjBox_{+}, \AdjDiamond_{-})^{*}$} \\
&  \Box_{+} a \leq \Box_{-} (\top \minus a) \tag*{$(\AdjBox_{-}, \AdjDiamond_{+})^{*}$}
\end{align*}

\end{figure}

  The equations which correspond to the forward locality conditions are shown in Figure \ref{fig:forward-locality-axioms}, and the equations which correspond to the backward locality conditions are shown in Figure \ref{fig:backward-locality-axioms}. The axioms $(\Diamond_{+}, \Box_{+})^{*}$ and $(\Box_{+}, \Diamond_{+})^{*}$ are precisely the axioms of positive modal algebras, while the axioms $(\Diamond_{+}, \Box_{+})^{*}$ and $(\AdjDiamond_{+}, \AdjBox_{+})^{*}$ are the axioms of intuitionistic modal algebras. Notice that the equations $(\Diamond_{-}, \Box_{-})^{*}$ and $(\Box_{-}, \Diamond_{-})^{*}$ are the only axioms among those considered in this paper which use both implication and co-implication.

  We first show that the backward locality axioms are obtained by applying the forward--backward duality to forward locality axioms. For example, the forward locality axiom $\Diamond_{+} a \rightarrow \Box_{+} b \leq \Box_{+} (a \rightarrow b)$ characteristic of intuitionistic modal algebras, is essentially the ordinary forward-looking axiom for backward-looking modalities $\AdjBox_{+} a \wedge \AdjDiamond_{+} b \leq \AdjDiamond_{+} (a \wedge b)$.

  We call an $\lang{L}$-modal algebra \emph{appropriate} for the equation $(\Delta_{\alpha}, \nabla_{\!\beta})^{*}$ if it is a Heyting, co-Heyting, or bi-Heyting algebra according to which connectives appear in the equation $(\Delta_{\alpha}, \nabla_{\!\beta})^{*}$. 

\begin{proposition} \label{prop:backward-locality-conditions}
  Let $\Delta'_{\alpha}$ and $\nabla'_{\!\beta}$ be adjoint to $\Delta_{\alpha}$ and $\nabla_{\!\beta}$, respectively. Let $\alg{A}$ be a tense $\lang{L}$-modal algebra appropriate for $(\Delta'_{\alpha}, \nabla'_{\!\beta})$. Then the adjoint of $\alg{A}$ satisfies $(\Delta_{\alpha}, \nabla_{\!\beta})^{*}$ if and only if $\alg{A}$ satisfies $(\Delta'_{\alpha}, \nabla'_{\!\beta})^{*}$.
\end{proposition}

\begin{proof}
  By order reversal it suffices to deal with the first two equations in each group.

  Suppose that the adjoint of $\alg{A}$ satisfies $(\Diamond_{+}, \Box_{+})^{*}$, i.e.\ that $\alg{A}$ satisfies $\AdjBox_{+} a \wedge \AdjDiamond_{+} b \leq \AdjDiamond_{+} (a \wedge b)$ for all $a, b \in \alg{A}$. Letting $a = \Diamond_{+} c$ and $b = \Diamond_{+} c \rightarrow \Box_{+} d$ (in particular $a \wedge b \leq \Box_{+} d$) yields $c \wedge \AdjDiamond_{+} b \leq \AdjBox_{+} \Diamond_{+} c \wedge \AdjDiamond_{+} b \leq \AdjDiamond_{+} (\Diamond_{+} c \wedge b) \leq \AdjDiamond_{+} \Box_{+} d \leq d$, hence $\AdjDiamond_{+} (\Diamond_{+} c \rightarrow \Box_{+} d) \leq c \rightarrow d$ and $\Diamond_{+} c \rightarrow \Box_{+} d \leq \Box_{+} c \rightarrow d$.

  Conversely, suppose that $\Diamond_{+} a \rightarrow \Box_{+} b \leq \Box_{+} (a \rightarrow b)$ for all $a, b \in \alg{A}$. But then $a \wedge \AdjDiamond_{+} (\Diamond_{+} a \rightarrow \Box_{+} b) \leq b$. Now $\AdjBox_{+} c \wedge \AdjDiamond_{+} d \leq \AdjBox_{+} c \wedge \AdjDiamond_{+} (c \rightarrow d) \leq \AdjBox_{+} c \wedge \AdjDiamond_{+} (\Diamond_{+} \AdjBox_{+} c \rightarrow \Box_{+} \AdjDiamond_{+} d) \leq \AdjDiamond_{+} d$, where the last inequality was obtained from $a \wedge \AdjDiamond_{+} (\Diamond_{+} a \rightarrow \Box_{+} b) \leq b$ by substituting $a = \AdjBox_{+} c$ and $b = \AdjDiamond_{+} d$.

  Suppose that the adjoint of $\alg{A}$ satisfies $(\Diamond_{-}, \Box_{-})^{*}$, i.e.\ that $\alg{A}$ satisfies $\AdjBox_{-} a \wedge \AdjDiamond_{-} b \leq \AdjDiamond_{-} (a \vee b)$. Letting $a = \Box_{-} c$ and $b = \Diamond_{-} d \minus \Box_{-} c$ (in particular $a \vee b \geq \Diamond_{-} d$) yields $c \wedge \AdjDiamond_{-} b \leq \AdjBox_{-} \Box_{-} c \wedge \AdjDiamond_{-} b \leq \AdjDiamond_{-} (a \vee b) \leq \AdjDiamond_{-} \Diamond_{-} d \leq d$, hence $\AdjDiamond_{-} b \leq c \rightarrow d$ and $\Diamond_{-} (c \rightarrow d) \leq \Diamond_{-} \AdjDiamond_{-} b \leq b \leq \Diamond_{-} d \minus \Box_{-} c$.

  Conversely, suppose that $\Diamond_{-} (a \rightarrow b) \leq \Diamond_{-} b \minus \Box_{-} a$ for all $a, b \in \alg{A}$. But then $a \wedge \AdjDiamond_{-} (\Diamond_{-} b \minus \Box_{-} a) \leq b$. Observe that the inequality $(c \vee d) \minus c \leq c \vee d$ holds for all $c, d \in \alg{A}$. It follows that $\AdjBox_{-} c \wedge \AdjDiamond_{-} (c \vee d) \leq \AdjBox_{-} c \wedge \AdjDiamond_{-} ((c \vee d) \minus c) \leq \AdjBox_{-} c \wedge \AdjDiamond_{-} (\Diamond_{-} \AdjDiamond_{-} (c \vee d) \minus \Box_{-} \AdjBox_{-} c) \leq \AdjDiamond_{-} (c \vee d)$, where the last inequality was obtained from $a \wedge \AdjDiamond_{-} (\Diamond_{-} b \minus \Box_{-} a) \leq b$ by substituting $a = \AdjBox_{-} c$ and $b = \AdjDiamond_{-} (c \vee d)$.

  Suppose that the adjoint of $\alg{A}$ satisfies $(\Diamond_{+}, \Box_{-})^{*}$, i.e.\ that $\alg{A}$ satisfies $\AdjBox_{-} a \wedge \AdjDiamond_{+} a \leq \bot$ for all $a \in \alg{A}$. Then $a \wedge \AdjDiamond_{+} \Box_{-} a \leq \AdjBox_{-} \Box_{-} a \wedge \AdjDiamond_{+} \Box_{-} a \leq \bot$, hence $\AdjDiamond_{+} \Box_{-} a \leq a \rightarrow \bot$ and $\Box_{-} a \leq \Box_{+} (a \rightarrow \bot)$.

  Conversely, let $x \leq \AdjDiamond_{+} a$ and $x \leq \AdjBox_{-} a$. Then $a \leq \Box_{-} x$, hence $x = x \wedge \AdjDiamond_{+} \Box_{-} x$. Therefore $x \leq \bot$ if and only if $\AdjDiamond_{+} \Box_{-} x \leq x \rightarrow \bot$, that is, if and only if $\Box_{-} x \leq \Box_{+} (x \rightarrow \bot)$.

  Suppose that the adjoint of $\alg{A}$ satisfies $(\Diamond_{-}, \Box_{+})^{*}$, i.e.\ that $\alg{A}$ satisfies $\AdjBox_{+} a \wedge \AdjDiamond_{-} a \leq \bot$ for all $a \in \alg{A}$. Then $a \wedge \AdjDiamond_{-} \Diamond_{+} a \leq \AdjBox_{+} \Diamond_{+} a \wedge \AdjDiamond_{-} \Diamond_{+} a \leq \bot$, hence $\AdjDiamond_{-} \Diamond_{+} a \leq a \rightarrow \bot$ and $\Diamond_{-} (a \rightarrow \bot) \leq \Diamond_{+} a$.

  Conversely, let $x \leq \AdjBox_{+} a$ and $x \leq \AdjDiamond_{-} a$. Then $\Diamond_{+} x \leq a$, hence $x = x \wedge \AdjDiamond_{-} \Diamond_{+} x$. Therefore $x \leq \bot$ if and only if $\AdjDiamond_{-} \Diamond_{+} x \leq x \rightarrow \bot$, that is, if and only if $\Diamond_{-} (x \rightarrow \bot) \leq \Diamond_{+} x$.
\end{proof}

\begin{theorem}[Correspondence] \label{thm:correspondence-for-locality-conditions}
  Each equation $(\Delta_{\alpha},\hskip -0.21em \nabla_{\!\beta})^{*}$ \hskip -0.12em corresponds to the frame condition $(\Delta_{\alpha},\hskip -0.21em \nabla_{\!\beta})^{*}$\!.
\end{theorem}

\begin{proof}
  The right-to-left direction is left to the reader. In the opposite direction, we start with the forward conditions. We only deal with the cases $(\Diamond_{\alpha}, \Box_{\alpha})^{*}$ and $(\Diamond_{\alpha}, \Box_{-\alpha})^{*}$, the remaining cases being order dual.

  Suppose first that $u \geq v (R^{\Diamond}_{\alpha} \cap R^{\Box}_{\alpha}) w$ but $u (R^{\Diamond}_{\alpha} \cap R^{\Box}_{\alpha}) x$ implies not$^{\alpha}$ $x \geq^{\alpha} w$. Then let $y \in^{\alpha} a$ if and only if $u R^{\Box}_{\alpha} y$ and let $y \in^{\alpha} b$ if and only if $y \geq^{\alpha} w$. It follows that $u \in \Box_{\alpha} a$ and $u \in \Diamond_{\alpha} b$. If $u \in \Diamond_{\alpha} (a \wedge^{\alpha} b)$, then there is some $z \geq^{\alpha} w$ such that $u R^{\Diamond}_{\alpha} z$ and $u R^{\Box}_{\alpha} z$, hence $u (R^{\Diamond}_{\alpha} \cap R^{\Box}_{\alpha}) z$, contradicting the assumption that $u (R^{\Diamond}_{\alpha} \cap R^{\Box}_{\alpha}) z$ implies not$^{\alpha}$ $z \geq^{\alpha} w$.

  Now suppose that $u R^{\Diamond}_{\alpha} v$ but not $u R^{\Box}_{\alpha} v$. Then let $w \in^{-\alpha} a$ if and only if $u R^{\Box}_{-\alpha} w$. It follows that $u \in \Box_{-\alpha} a$ and $u \in \Diamond_{\alpha} a$, contradicting $\Box_{-\alpha} a \wedge \Diamond_{\alpha} a \leq \bot$.

  To deal wih the backward conditions, we appeal to Proposition \ref{prop:backward-locality-conditions}. Consider a backward condition $(\Delta'_{\alpha}, \nabla'_{\!\beta})^{*}$. Let $\Delta'_{\alpha}$ and $\nabla'_{\!\beta}$ be the adjoints of $\Delta_{\alpha}$ and $\nabla_{\!\beta}$, respectively. A modal frame $\modframe{F}$ satisfies the backward condition $(\Delta'_{\alpha}, \nabla'_{\!\beta})^{*}$ if and only if its converse $\converse{\modframe{F}}$ satisfies the forward condition $(\Delta_{\alpha}, \nabla_{\!\beta})^{*}$, i.e.\ by Proposition \ref{prop:symmetries} if and only if the adjoint of the complex algebra $\modframe{F}^{+}$ satisfies the equation $(\Delta_{\alpha}, \nabla_{\!\beta})^{*}$. By Proposition \ref{prop:backward-locality-conditions} this holds if and only if $\modframe{F}^{+}$ satisfies $(\Delta'_{\alpha}, \nabla'_{\!\beta})^{*}$\!.
\end{proof}

  The equations $(\Delta_{\alpha}, \nabla_{\!\beta})^{*}$ imply the corresponding quasi-equations $(\Delta_{\alpha}, \nabla_{\!\beta})$. This is why the conditions $(\Delta_{\alpha}, \nabla_{\!\beta})$ do not appear in the definition of intuitionistic modal algebras and positive modal algebras: they have been made redundant by the conditions $(\Delta_{\alpha}, \nabla_{\!\beta})^{*}$\!.

\begin{proposition} \label{prop:star-is-stronger}
  Let $\Delta'_{\alpha}$ and $\nabla'_{\!\beta}$ be backward modalities adjoint to the forward modalities $\Delta_{\alpha}$ and $\nabla_{\!\beta}$. Then the equation $(\Delta_{\alpha}, \nabla_{\!\beta})^{*}$ implies the quasi-equation $(\Delta_{\alpha}, \nabla_{\!\beta})$, and so does the equation $(\Delta'_{\alpha}, \nabla'_{\!\beta})^{*}$.
\end{proposition}

\begin{proof}
  We omit the easy proof of this observation.
\end{proof}

  It remains to show that the equations $(\Delta_{\alpha}, \nabla_{\!\beta})^{*}$ are canonical.

\begin{theorem}[Canonicity] \label{thm:canonicity-for-locality-conditions}
  All of the equations $(\Delta_{\alpha}, \nabla_{\!\beta})^{*}$ are canonical.
\end{theorem}

\begin{proof}
  The forward locality axioms are canonical by virtue of being Sahlqvist in the sense of \cite{gehrke+nagahashi+venema05}. However, for the sake of being self-contained we provide a brief proof of this fact. By order duality it suffices to consider the equations $(\Diamond_{\alpha}, \Box_{\alpha})^{*}$ and $(\Diamond_{\alpha}, \Box_{-\alpha})^{*}$. We proceed as in Theorem \ref{thm:canonicity}.

  Suppose that $\filter{U} R^{\Diamond}_{+} \filter{V}$, i.e.\ $\Diamond_{+} [\filter{V}] \subseteq \filter{U}$. Extend $\filter{V}$ to a maximal filter $\filter{W}$ such that $\Diamond_{+} [\filter{W}] \subseteq \filter{U}$. The filter $\filter{W}$ is prime and if $a \notin \filter{W}$, then $\Diamond_{+} (a \wedge w) \notin \filter{U}$ for some $w \in \filter{W}$. But then $\Diamond_{+} w \in \filter{U}$, hence $\Box_{+} a \notin \filter{U}$ because $\Box_{+} a \wedge \Diamond_{+} w \leq \Diamond_{+} (a \wedge w)$. Thus $\filter{U} (R^{\Diamond}_{+} \cap R^{\Box}_{+}) \filter{W} \supseteq \filter{V}$. The proof for $(\Diamond_{-}, \Box_{-})$ entirely is analogous.

  Suppose that $\filter{U} R^{\Diamond}_{\alpha} \filter{V}$, i.e.\ $\Diamond_{\alpha} [\filter{V}] \subseteq \filter{U}$. If $a \in^{\alpha} \filter{V}$, then $\Diamond_{\alpha} a \in \filter{U}$, hence $\Box_{-\alpha} a \notin \filter{U}$ because $\Box_{-\alpha} a \wedge \Diamond_{\alpha} a \leq \bot$. We therefore have $\filter{U} R^{\Box}_{-\alpha} \filter{V}$.

  It remains to deal with the backward locality axioms. By order duality it suffices to consider the equations $(\AdjDiamond_{\alpha}, \AdjBox_{\alpha})^{*}$ and $(\AdjDiamond_{\alpha}, \AdjBox_{-\alpha})^{*}$.

  Suppose that $\filter{U} R^{\Box}_{+} \filter{V}$, i.e.\ $\Box_{+} [\alg{A} \setminus \filter{V}] \subseteq \alg{A} \setminus \filter{U}$. Extend $\filter{U}$ to a maximal filter $\filter{W}$ disjoint from $\Box_{+} [\alg{A} \setminus \filter{V}]$. The filter $\filter{W}$ is prime, $\filter{W} R^{\Box}_{+} \filter{V}$, and if $\Diamond_{+} a \notin \filter{W}$, then $w \wedge \Diamond_{+} a \leq \Box v$ for some $w \in \filter{W}$ and $v \notin \filter{V}$. Thus $w \leq \Diamond_{+} a \rightarrow \Box_{+} v \leq \Box_{+} (a \rightarrow v)$, so $a \rightarrow v \in \filter{V}$ and $a \notin \filter{V}$, showing that $\filter{U} \subseteq \filter{W} (R^{\Box}_{+} \cap R^{\Diamond}_{+}) \filter{V}$.

  Suppose that $\filter{U} R^{\Diamond}_{-} \filter{V}$, i.e.\ $\Diamond_{-} [\alg{A} \setminus \filter{V}] \subseteq \filter{U}$. Extend $\alg{A} \setminus \filter{U}$ to a maximal ideal $\ideal{I}$ disjoint from $\Diamond_{-} [\alg{A} \setminus \filter{V}]$, and let $\filter{W} = \alg{A} \setminus \ideal{I}$. The filter $\filter{W}$ is prime, $\filter{W} R^{\Diamond}_{-} \filter{V}$, and if $\Box_{-} a \in \filter{W}$, then $\Diamond_{-} v \leq w \vee \Box_{-} a$ for some $v \notin \filter{V}$ and $w \notin \filter{W}$. Thus $\Diamond_{-} (a \rightarrow v) \leq \Diamond_{-} v \minus \Box_{-} a \leq w \notin \filter{W}$, so $a \rightarrow v \in \filter{V}$ and $a \notin \filter{V}$, showing that $\filter{U} \supseteq \filter{W} (R^{\Diamond}_{-} \cap R^{\Box}_{-}) \filter{U}$.

  Suppose that $\filter{U} R^{\Box}_{+} \filter{V}$, i.e.\ $\Box_{+} [\alg{A} \setminus \filter{V}] \subseteq \alg{A} \setminus \filter{U}$, and $a \in \filter{V}$. If $\Box_{-} a \in \filter{U}$, then we have $\Box_{+} (a \rightarrow \bot) \in \filter{U}$ because $\Box_{-} a \leq \Box_{+} (a \rightarrow \bot)$, so $a \rightarrow \bot \in \filter{V}$ and $\bot \in \filter{V}$. But $\bot \notin \filter{V}$, hence $\Box_{-} a \notin \filter{U}$ and $\filter{U} R^{\Box}_{-} \filter{V}$.

  Suppose that $\filter{U} R^{\Diamond}_{-} \filter{V}$, i.e.\ $\Diamond_{-}[\alg{A} \setminus \filter{V}] \subseteq \filter{U}$, and $a \in \filter{V}$. If $\Diamond_{+} a \notin \filter{U}$, then we have $\Diamond_{-} (a \rightarrow \bot) \notin \filter{U}$ because $\Diamond_{-} (a \rightarrow \bot) \leq \Diamond_{+} a$, so $a \rightarrow \bot \in \filter{V}$ and $\bot \in \filter{V}$. But $\bot \notin \filter{V}$, hence $\Diamond_{+} a \in \filter{U}$ and $\filter{U} R^{\Diamond}_{+} \filter{V}$.
\end{proof}

  We call an $\lang{L}$-unimodal algebra \emph{appropriate} for $(\Delta_{\alpha})^{*}$ if it is appropriate for $(\Delta_{\alpha}, \nabla_{\alpha})^{*}$ in case $\nabla_{\alpha} \in \lang{L}$ and also appropriate for $(\Delta_{\alpha}, \nabla_{-\alpha})^{*}$ in case $\nabla_{-\alpha} \in \lang{L}$. We say that an $\lang{L}$-modal algebra appropriate for $(\Delta_{\alpha})^{*}$ satisfies $(\Delta_{\alpha})^{*}$ if it satisfies $(\Delta_{\alpha}, \nabla_{\alpha})^{*}$ in case $\nabla_{\alpha} \in \lang{L}$ and also satisfies $(\Delta_{\alpha}, \nabla_{-\alpha})^{*}$ in case $\nabla_{-\alpha} \in \lang{L}$. Given a tense unimodal signature $\lang{L}'$, an \emph{$\lang{L}'$-local} $\lang{L}$-unimodal algebra is an $\lang{L}$-unimodal algebra appropriate for $(\Delta_{\alpha})^{*}$ for each $\Delta_{\alpha} \in \lang{L}'$ which satisfies each such $(\Delta_{\alpha})^{*}$. Similarly, an \emph{$\lang{L}'$-local $\lang{L}$-unimodal frame} is an $\lang{L}$-unimodal frame which satisfies $(\Delta_{\alpha})^{*}$ for each $\Delta_{\alpha} \in \lang{L}'$.

  Combining the canonicity and correspondence results of the last two sections now yields the following Kripke completeness theorem for the logic of $\lang{L}$-unimodal frames satisfying any combination of the forward and backward locality conditions, provided that the appropriate modalities are present in $\lang{L}$.

\begin{theorem}[Kripke completeness]
  The quasivariety of $\lang{L}'$-local $\lang{L}$-unimodal algebras is generated as a universal class by the class of all complex algebras of $\lang{L}'$-local $\lang{L}$-unimodal frames.
\end{theorem}

   Let us now consider an example to illustrate how to obtain an axiomatization of such a logic from the menu of axioms presented above.

\begin{example} \label{example}
  Suppose that we want to axiomatize the quasivariety of Heyting $\lang{L}$-modal algebras generated by the complex algebras of all $\lang{L}$-unimodal frames for $\lang{L} = \{ \Box_{+}, \Diamond_{+}, \Box_{-} \}$ which satisfy the conditions
\begin{align*}
  {\leq} \circ {R} & \subseteq {R} \circ {\geq}, \\
  {R} \circ {\leq} & \subseteq {\leq} \circ {R}.
\end{align*}
  Looking at our list of locality conditions, we see that these are the conditions $(\Box_{-})^{*}$ and $(\AdjDiamond_{+})^{*}$, therefore we first need to check that $\Box_{-} \in \lang{L}$ and $\AdjDiamond_{+} \in \converse{\lang{L}}$, i.e.\ $\Box_{+} \in \lang{L}$. The next step is to inspect Figures~\ref{fig:forward-locality-axioms} and~\ref{fig:backward-locality-axioms} to find all conditions of the form $(\Box_{-}, \Delta_{\alpha})^{*}$ for $\Delta_{\alpha} \in \lang{L}$ and $(\AdjDiamond_{+}, \nabla_{\!\beta})^{*}$ for $\nabla_{\!\beta} \in \converse{L}$. This yields the conditions $(\Box_{-}, \Diamond_{+})^{*}$ and $(\AdjDiamond_{+}, \AdjBox_{+})^{*}$ and $(\AdjDiamond_{+}, \AdjBox_{-})^{*}$. The conditions $(\AdjDiamond_{+}, \AdjBox_{+})^{*}$ and $(\AdjDiamond_{+}, \AdjBox_{-})^{*}$ require the Heyting implication, therefore we need to check that Heyting implication is present in our signature.

  The locality axioms $(\Box_{-}, \Diamond_{+})^{*}$ and $(\AdjDiamond_{+}, \AdjBox_{+})^{*}$ and $(\AdjDiamond_{+}, \AdjBox_{-})^{*}$ imply the ordinary axioms $(\Box_{-}, \Diamond_{+})$ and $(\Box_{+}, \Diamond_{+})$ and $(\Box_{+}, \Box_{-})$, therefore the latter do not need to be part of our axiomatization. However, every other condition $(\Delta_{\alpha}, \nabla_{\!\beta})$ for $\Delta_{\alpha}, \nabla_{\!\beta} \in \lang{L}$ will be part of our axiomatization. In particular, these are the conditions $(\Box_{-}, \Box_{+})$ and $(\Diamond_{+}, \Box_{+})$ and $(\Diamond_{+}, \Box_{-})$. We also recall that the condition $(\Diamond_{+}, \Box_{+})$ may be stated equationally if Heyting implication is part of the signature.

  We thus obtain an axiomatization which contains the three locality equations $(\Box_{-}, \Diamond_{+})^{*}$ and $(\AdjDiamond_{+}, \AdjBox_{+})^{*}$ and $(\AdjDiamond_{+}, \AdjBox_{-})^{*}$, the two equations $(\Box_{-}, \Box_{+})$ and $(\Diamond_{+}, \Box_{+})$, and the quasiequation $(\Diamond_{+}, \Box_{-})$. Had we moreover imposed the condition $(\Diamond_{+})^{*}$, the equation $(\Diamond_{+}, \Box_{+})$ and the quasiequation $(\Diamond_{+}, \Box_{-})$ would be replaced by the equations $(\Diamond_{+}, \Box_{+})^{*}$ and $(\Diamond_{+}, \Box_{-})^{*}$.
\end{example}

  Finally, we characterize the $\lang{L}$-modal algebras which are the negation-free subreducts of Boolean algebras with an operator. In the case of $\lang{L} = \{ \Box_{+}, \Diamond_{+} \}$ these are precisely the positive modal algebras of Dunn~\cite{dunn95}.

\begin{definition}
  A Boolean $\lang{L}$-modal algebra is an $\lang{L}$-modal algebra expanded by a Boolean negation $\neg$ which satisfies all applicable equations from the following (redunant) list:
\begin{align*}
  \Box_{+} a & = \Box_{-} \neg a, & \Box_{+} a & = \neg \Diamond_{+} \neg a, & \Box_{+} a & = \neg \Diamond_{-} a, \\
  \Box_{-} a & = \Box_{+} \neg a, & \Box_{-} a & = \neg \Diamond_{+} a, & \Box_{-} a & = \neg \Diamond_{-} \neg a, \\
  \Diamond_{+} a & = \Diamond_{-} \neg a, & \Diamond_{+} a & = \neg \Box_{+} \neg a, & \Diamond_{+} a & = \neg \Box_{-} a, \\
  \Diamond_{-} a & = \Diamond_{+} \neg a, & \Diamond_{-} a & = \neg \Box_{+} a, & \Diamond_{-} a & = \neg \Box_{-} \neg a.
\end{align*}
\end{definition}

  Because all modalities are interdefinable, Boolean $\lang{L}$-modal algebras for any choice of $\lang{L}$ are termwise equivalent to ordinary Boolean algebras with a modal operation, i.e.\ to Boolean algebras with a $\Box_{+}$ operator.

\begin{definition}
  A \emph{positive $\lang{L}$-modal algebra} is an $\lang{L}$-modal algebra which satisfies the forward condition $(\Delta_{\alpha})^{*}$ for each $\Delta_{\alpha} \in \lang{L}$.
\end{definition}

  By Proposition~\ref{prop:star-is-stronger} positive $\lang{L}$-modal algebras are $\lang{L}$-unimodal algebras.

\begin{theorem}[Positive modal algebras] \label{thm:positive-modal-logic}
  Let $\alg{A}$ be an $\lang{L}$-modal algebra. Then $\alg{A}$ is a subreduct of a Boolean $\lang{L}$-modal algebra if and only if $\alg{A}$ is a positive $\lang{L}$-modal algebra.
\end{theorem}

\begin{proof}
  For the left-to-right direction, it suffices to verify that the forward locality equation $(\Delta_{\alpha})^{*}$ holds in each $\lang{L}$-modal algebra for $\Delta_{\alpha} \in \lang{L}$ and the backward locality equation $(\nabla_{\!\beta})^{*}$ holds in each tense $\lang{L}$-modal algebra for $\nabla_{\!\beta} \in \lang{L}$. Right-to-left, let $\alg{A}$ be a positive (tense) $\lang{L}$-modal algebra. By Theorem \ref{thm:canonicity-for-locality-conditions} its canonical extension $\complex{\canframe{\alg{A}}}$ is also a positive (tense) $\lang{L}$-modal algebra which contains $\alg{A}$ as a subalgebra. In particular, $\canframe{\alg{A}}$ is an $\lang{L}$-unimodal frame with $R = \bigcap_{\Delta_{\alpha} \in \lang{L}} R^{\Delta}_{\alpha}$. Let $\alg{B}$ be the expansion of the Boolean algebra of all subsets of $\canframe{\alg{A}}$ (not necessarily upsets) by the following operations for the modalities in $\lang{L}$:
\begin{align}
  \Box_{+} a & = \set{u \in W}{u R v \text{ implies } v \in a} \notag\\
  \Box_{-} a & = \set{u \in W}{u R v \text{ implies } v \notin a} \notag\\
  \Diamond_{+} a & = \set{u \in W}{u R v \text{ for some } v \in a} \notag\\
  \Diamond_{-} a & = \set{u \in W}{u R v \text{ for some } v \notin a} \notag
\end{align}
In case $\alg{A}$ is a tense $\lang{L}$-modal algebra, we also need to expand $\canframe{\alg{A}}$ by the appropriate backward modalities. Propositions \ref{prop:local-modalities} and \ref{prop:backward-local-modalities} now imply that $\complex{\canframe{\alg{A}}}$ is a subalgebra of $\alg{B}$.
\end{proof}

\end{document}